\documentclass[12pt]{iopart}
 \expandafter\let\csname equation*\endcsname\relax
 \expandafter\let\csname endequation*\endcsname\relax
\usepackage{bm}
\usepackage{enumerate}
\usepackage[top=1in, bottom=1in, left=1in, right=1in]{geometry}
\usepackage{mathrsfs,color}
\usepackage{amsfonts}\usepackage{multirow}
\usepackage{amssymb}
\usepackage{caption,comment,dsfont}
\usepackage{blkarray}
\usepackage{amsmath}
\usepackage{float}
\usepackage{amssymb,amsthm}
\usepackage[toc,page]{appendix}
\usepackage{graphicx,afterpage}
\usepackage{epstopdf}
\usepackage{mathrsfs,color}
\usepackage{amsfonts}\usepackage{multirow}
\usepackage{amssymb}
\usepackage{dsfont}
\usepackage{caption,comment}
\usepackage{blkarray}
\usepackage{amsmath}
\usepackage{float}
\usepackage{amssymb,amsthm}
\usepackage[toc,page]{appendix}
\usepackage{graphicx,afterpage}
\usepackage{epstopdf}
\usepackage{cancel}
\usepackage{mathdots}
\usepackage[pdftex,colorlinks=true]{hyperref}
\usepackage[pdftex,colorlinks=true]{hyperref}
\hypersetup{CJKbookmarks,%
bookmarksnumbered,%
colorlinks,%
linkcolor=black,%
citecolor=black,%
plainpages=false,%
pdfstartview=FitH}
\numberwithin{equation}{section}
\numberwithin{figure}{section}
\def\theequation{\arabic{section}.\arabic{equation}}
\newcommand\tabcaption{\def\@captype{table}\caption}
\def\thefigure{\arabic{section}.\arabic{figure}}
\newtheorem{thm}{Theorem}[section]
\newtheorem{cor}[thm]{Corollary}
\newtheorem{lem}[thm]{Lemma}

\newtheorem{aspt}[thm]{Assumption}
\newtheorem{rem}[thm]{Remark}
\newtheorem{example}[thm]{Example}

\usepackage{mhequ}
\hypersetup{CJKbookmarks,%
bookmarksnumbered,%
colorlinks,%
linkcolor=black,%
citecolor=black,%
plainpages=false,%
pdfstartview=FitH}
\numberwithin{equation}{section}
\numberwithin{figure}{section}

\newcommand{\CX}{\mathcal{X}}
\newcommand{\CY}{\mathcal{Y}}
\newcommand{\CB}{\mathcal{B}}
\newcommand{\CD}{\mathcal{D}}

\newcommand{\Ustar}{U^*}
\newcommand{\Zstar}{Z^*}

\newcommand{\Vhat}{\widehat{V}}
\newcommand{\Chat}{\widehat{C}}

\newcommand{\CE}{\mathcal{E}}
\newcommand{\CM}{\mathcal{M}}

\newcommand{\BE}{\mathbb{E}}
\newcommand{\reals}{\mathbb{R}}
\newcommand{\norm}[1]{\| #1 \|}
\newcommand{\unit}{\mathds{1}}

\newcommand{\Mmax}{M_{0}}
\newcommand{\Sigmamax}{D_{0}}
\newcommand{\Lambdamax}{\Lambda_{0}}
\newcommand{\Qmax}{Q_{0}}
\newcommand{\Rmax}{R_{0}}

\newcommand{\vbar}{\bar{v}}
\newcommand{\Vbar}{\bar{V}}

\newcommand{\Shat}{\widehat{S}}

\newcommand{\Vhatbar}{\overline{\widehat{V}}}

\begin{document}
\title{Nonlinear stability and ergodicity  of ensemble based Kalman filters}
\author{Xin T Tong$^{1,2}$, Andrew J  Majda$^{1,2}$  and David Kelly$^1$ }
\address{$^1$Courant Institute of Mathematical Sciences, New York University, New York, NY 10012, USA}
\address{$^2$Center for Atmosphere Ocean Science, New York University, New York, NY 10012, USA}
\date{\today}
\eads{
\mailto{tong@cims.nyu.edu}
}

\begin{abstract}
The ensemble Kalman filter (EnKF) and ensemble square root filter (ESRF) are data assimilation methods used to combine high dimensional, nonlinear dynamical models with observed data. Despite their widespread usage in climate science and oil reservoir simulation, very little is known about the long-time behavior of these methods and why they are effective when applied with modest ensemble sizes in large dimensional turbulent dynamical systems.  By following the basic principles of energy dissipation and controllability of filters, this paper  establishes a simple, systematic and rigorous framework for the nonlinear analysis of EnKF and ESRF with arbitrary ensemble size, focusing on the dynamical properties of boundedness and geometric ergodicity. The time uniform boundedness guarantees that the filter estimate will not diverge to machine infinity in finite time, which is a potential threat for EnKF and ESQF known as the catastrophic filter divergence. Geometric ergodicity ensures in addition that the filter has a unique invariant measure and  that initialization errors will dissipate exponentially in time. We establish these results by introducing a natural notion of observable energy dissipation. The time uniform bound is achieved through a simple Lyapunov function argument, this result applies to systems with complete observations and strong kinetic energy dissipation, but also to concrete examples with incomplete observations. With the Lyapunov function argument established, the geometric ergodicity is obtained by verifying the controllability of the filter processes; in particular,  such analysis for ESQF relies on a careful multivariate perturbation analysis of the covariance eigen-structure. 
\end{abstract}

\section{Introduction}
An important problem in scientific computing is the effective assimilation of observational data with high dimensional nonlinear forecast models. 
The classical filtering tools, such as the Kalman filter and extended Kalman filter, are poorly suited to these problems, due both to the nonlinearity of the models and the cost of computing covariance matrices for high dimensional state vectors.   
The ensemble Kalman filter (EnKF) and ensemble square root filter (ESRF) were designed to overcome these difficulties \cite{evensen03,bishop01,And01, MH12}. 
The basic idea of these methods is to propagate an ensemble $\{V^{(1)}_n,\ldots, V^{(K)}_n\}$ to describe the forecast distribution of the underlying system $U_n$, and then assimilating the new observation via a Kalman-type update using the ensemble mean and covariance. The state estimate remains useful even when the ensemble size is several orders of magnitude smaller than the state dimension, leading to a considerable benefit in computational cost. Due to their efficiency, EnKF and ESRF are broadly used, notably in ocean-atmosphere science \cite{kalnay03,MH12} and oil reservoir simulations \cite{navdal05}. 

Despite the ubiquitous application of EnKF and ESRF, little is known of their dynamical behavior beyond that provided by numerical experiments. Existing theoretical studies of EnKF and ESRF focus mainly on either error estimation of one single assimilation step \cite{FB07, LDN08}, or in the case of linear model dynamics, convergence to the classical Kalman filter as the ensemble size tends to infinity \cite{legland10,mandel2011convergence}. The aim of this article is to address a more practical scenario, namely by looking at the long-time behavior of the ensemble when the ensemble size is fixed and where the underlying model is nonlinear.  
To be specific, we seek to address the following questions:
\begin{enumerate}
\item Under what model conditions does the ensemble remain bounded on an infinite time horizon?
\item Are the filter processes ergodic and how quickly do they lose memory of initial conditions?
\end{enumerate}
These two questions are of great practical importance. Boundedness of the ensemble prohibits the state estimate from diverging to infinity, thereby precluding the disastrous phenomena of catastrophic filter divergence \cite{MH08, HM10, GM13, MH12}. Ergodicity of the ensemble ensures that errors in the initialization of the filter will not affect the performance of the filter in the long run \cite{CR11, TvH12, TvH14}, geometric ergodicity further insures that the error will dissipate exponentially fast in time.  
To the best of our knowledge, the only article in a similar setting is \cite{KLS14}, where the authors show well-posedness of EnKF with bounds which can grow exponentially in time and accuracy under variance inflation.
%

In the study of dynamical systems, 
boundedness can be demonstrated through the construction of a Lyapunov function $\mathcal{E}$, which is a positive function statisfying the dissipation criterion
\begin{equ}\label{e:intro_lyap}
\mathbb{E}_{n-1} \mathcal{E}(U_n)\leq (1-\beta) \mathcal{E}(U_{n-1})+K\;.
\end{equ}
Here $\mathbb{E}_{n-1}$ denotes the conditional expectation with respect to the information at time $n-1$, and $0 <  \beta < 1$ is a constant. Using the discrete Gr\"{o}nwall inequality, we immediately find that 
\[
\mathbb{E} \mathcal{E}(U_n)\leq (1-\beta)^n\mathbb{E} \mathcal{E}(U_0)+K\beta^{-1},
\]
which shows that, under expectation of the Lyapunov function $\CE$, the state $U_n$ is bounded uniformly in $n$.  
In geophysically relevant models, such as the Lorenz equations and Navier-Stokes equations, the corresponding $\mathcal{E}$ can be chosen as the kinetic energy $\CE (\cdot)= |\cdot|^2$. In this scenario, the relation \eqref{e:intro_lyap} is known as an energy principle. 
\par
A natural strategy for proving boundedness of EnKF and ESRF is to check whether the energy principles of the nonlinear system are inherited by the ensemble. In other words, if $\mathcal{E}$ is a Lyapunov function of the original system $U_n$, can we use $\mathcal{E}$ to construct a Lyapunov function for the ensemble processes $\{V^{(k)}_n\}_{k=1}^K$. 
\par 
In Theorems \ref{thm:EnKFobs} and \ref{thm:EAKF} we will show that this construction is quite straight-forward, provided that the underlying model satisfies a so-called observable energy criterion. To be specific, suppose the model is observed linearly via $Z_n = HU_n + \zeta_n$, then the observable energy criterion states that the underlying model satisfies \eqref{e:intro_lyap} with the choice $\mathcal{E}(u)=|Hu|^2$, that is
\begin{equ}\label{e:intro_OEC}
\BE_{n-1} |H U_n|^2 \leq (1-\beta) |H U_{n-1}|^2 + K\;.
\end{equ}
 Under this assumption, it is shown in Theorems \ref{thm:EnKFobs}, \ref{thm:EAKF} that the ensemble $\{V_n^{(k)}\}_{k=1}^K$ satisfies a related energy principle.
%
%
%
%
Hence, if the model satisfies \eqref{e:intro_OEC} with full rank $H$, then the ensemble $\{V_n^{(k)}\}$ must remain bounded on an infinite time horizon. When $H$ is not of full rank, one still obtains an energy principle from \eqref{e:intro_OEC}, but can only conclude boundedness of the observable ensemble $\{HV^{(k)}\}_{k=1}^K$. 
\par


\par
With a Lyapunov function established, the EnKF and ESRF are shown to be geometrically ergodic by Theorems \ref{thm:EnKFergo}, \ref{thm:ETKFergo} and \ref{thm:EAKFergo}, assuming the nonlinear system is propagated with non-degenerate noise. The proofs are conceptually simple, as it suffices to check to the controllability of the filters, thanks to the classical work of \cite{MT93, MS02, MSH02}. The only technical challenge lies in the eigenvalue decomposition of matrices required in the assimilation step of ESRF. This  can be resolved by a careful multivariate perturbation analysis of the underlying matrices. 
\par
With  a short discussion in Section \ref{sec:range}, we will demonstrate a few sufficient conditions that imply the observable energy criterion \eqref{e:intro_OEC}. When $H$ is of full rank, the observable energy criterion holds provided that the dynamics have an energy principle with strong contraction parameter, depending on the condition number of $H$. When $H$ is not of full rank, the observable energy criterion does not hold in general, but is verifiable in several concrete examples through direct calculation. This dichotomy of observational rank  agrees with known numerical evidence, where the ensemble behaves stably when full rank observations are available, but can experience filter divergence when the observations are sparse \cite{HM10, MH12}, or even reach machine infinity in finite time, which is known as  catastrophic filter divergence \cite{MH08, MH12, GM13}. Using the same philosophy in this paper, the authors have found a concrete dynamical system that satisfies the kinetic energy criterion but not the observable energy criterion, and whose EnKF ensemble (provably) experiences catastrophic filter divergence with large probability. The authors have also found a general adaptive covariance inflation scheme, which always insures that the ensemble remains bounded on an infinite time horizon, without hurting the accuracy of original filters. These results will be reported in two separate papers \cite{KMT15, TMK15}.

\par
The article is structured as follows. In Section \ref{sec:model} we formulate the EnKF and ESRF methods and also introduce the notion of Lyapunov functions and energy principles. In Section \ref{sec:stabobs} we establish a simple framework to verify energy principles for EnKF and ESRF using the observable energy. Section \ref{sec:range} discusses the applicability of this framework by studying a few sufficient conditions that guarantee the observable energy condition. In Section \ref{sec:ergo} we prove the geometric ergodicity of the filter processes assuming the stability results in Section \ref{sec:stabobs} hold. In Section \ref{sec:conclude} we conclude this paper and discuss possible extensions. 

\section{Models setup and fundamental concepts}
\subsection{Model setup}
\label{sec:model}
In this paper, we assume the signal sequence $U_n\in \mathbb{R}^d$  is generated through a nonlinear mapping $\Psi_h$ plus a mean zero noise $\zeta_n$, and the observation is a linear one plus some mean zero noise in $\mathbb{R}^q$:
\begin{equation}
\label{sys:discrete}
U_{n}=\Psi_h(U_{n-1})+\zeta_n, \quad Z_n=HU_n+\xi_n.
\end{equation}
Here $\{\xi_n\}$ is an i.i.d. noise sequence, and $\zeta_n$ is independent of $\zeta_{1},\ldots \zeta_{n-1}$ conditioned on the realization of $U_{n-1}$. In many cases, the model may be generated by the solution of a stochastic differential equation (SDE) 
\begin{equation}
\label{sys:flow}
du_t=\psi(u_t)dt+\Sigma dW_t 
\end{equation}
by taking $U_n = u_{nh}$ for some fixed $h>0$. A short discussion of this discrete time formulation is attached in \ref{sec:discrete}.
%

At this stage, we impose no restrictions on $\zeta_n$ except that it is mean zero with a conditional covariance depends on $U_{n-1}$:
\begin{equation}
\label{eqn:Rh}
\mathbb{E}(\zeta_n |U_{n-1})=0, \quad \mathbb{E}(\zeta_n\otimes \zeta_n|U_{n-1})=R_h(U_{n-1}).
\end{equation}

As for the observation part, we will assume in this paper that
\[
\text{rank}(H)=q\leq d,\quad \mathbb{E}(\xi_n |U_{n-1})=0, \quad \mathbb{E}(\xi_n\otimes \xi_n|U_{n-1})=I_q. 
\]
The seemingly restrictive choice of observational noise covariance can be made without loss of generality. Indeed, any observational covariance can be reduced to the identity via a simple coordinate change on the filtering problem. To apply the results of this article to a filtering problem with non-trivial observational covariance, one must first apply the coordindate change and then check the assumptions in the new system of coordinates.  Details are contained in the remark below. 

\begin{rem}\label{rmk:svd}
Suppose that $\xi_n$ has a nonsingular covariance matrix $\Gamma$ and  $\Gamma^{-1/2}H$ has an SVD decomposition $\Gamma^{-1/2}H=\Phi \Lambda \Psi^T$, then we change the coordinate system and consider 
\begin{equation}
\label{eqn:HSVD}
\widetilde{U}_n=\Psi^TU_n,\quad \tilde{\xi}_n=\Phi^T\Gamma^{-1/2}\xi_n,\quad
\widetilde{Z}_n=\Phi^T \Gamma^{-1/2}Z_n=\Lambda \widetilde{U}_n+\tilde{\xi}_n. 
\end{equation}
Hence this change of coordinates also reduces the observation matrix to a diagonal matrix. If the observation dimension $q$ is larger than the model dimension $d$, the last $d-q$ diagonal entries of $\Lambda$ are zero, so the last $d-q$ rows of $\widetilde{Z}_n$ are independent of the signal and useless for filtering purpose, which we can ignore and set $d=q$. Moreover $\tilde{\xi}_n$ will have covariance matrix $I_q$. Since all the transformations above are linear and  bijective,  filtering $\widetilde{U}_n$ with $\widetilde{Z}_n$ is equivalent to filtering $U_n$ with $Z_n$. 
On the other hand, if the covariance $\Gamma$ is singular, then certain linear subspace can be observed exactly, and may cause the filtering operation to be singular. We do not consider such pathological cases in this paper. 
\end{rem}

\subsection{Ensemble Kalman filter}
In the standard Kalman filtering theory, the conditional distribution of the signal  process $U_n$  given the observation sequence $Z_1,\ldots, Z_n$ is given by a Gaussian distribution. EnKF inherits this idea by using a group of ensembles $\{V_n^{(k)}\}_{k=1}^K$ to represent this Gaussian distribution, as the mean and covariance can be taken as the ensemble mean and covariance. The EnKF operates very much like a Kalman filter, except its forecast step requires a Monte Carlo simulation due to the nonlinearity of the system. In detail, the EnKF is an iteration of following two steps, with (for instance) $\widehat{V}^{(k)}_0$ being sampled from the equilibrium measure of $U_n$.
\begin{itemize}
\item Forecast step: from the posterior ensemble at time $n-1$, $\{V^{(k)}_{n-1}\}_{k=1}^K$,
a forecast ensemble for time $n$ is generated by 
\[
\widehat{V}_{n}^{(k)}=\Psi_h(V_{n-1}^{(k)})+\zeta^{(k)}_n, 
\]
where $\zeta^{(k)}_n$ are independent samples drawn from the same distribution as $\zeta_n$. Then the prior distribution for time $n$ is described by the ensemble mean and covariance:
\begin{equation}
\label{eqn:C}
\overline{\widehat{V}}_{n}:=\frac{1}{K}\sum^K_{k=1} \widehat{V}^{(k)}_{n},\quad \widehat{C}_{n}:=\frac{1}{K-1}\sum_{k=1}^K(\widehat{V}^{(k)}_{n}-\overline{\widehat{V}}_{n})\otimes (\widehat{V}^{(k)}_{n}-\overline{\widehat{V}}_{n}). 
\end{equation}
\item Analysis step: upon receiving the new observation $Z_n$, random perturbations of it are generated by adding $\xi^{(k)}_{n}$:
\[
Z^{(k)}_{n}=Z_{n}+\xi^{(k)}_{n},
\]
where $\xi^{(k)}_n$ are independent samples drawn from the same distribution as $\xi_{n}$. Each ensemble member is then updated to 
\[
V^{(k)}_{n}=\widehat{V}^{(k)}_{n}-\widehat{C}_{n}H^T(I+H\widehat{C}_nH^T)^{-1}(H \widehat{V}^{(k)}_n-Z^{(k)}_{n})\;.
\]
\end{itemize}
In summary, the EnKF is generated by the following dynamics
\begin{equation}
\label{sys:EnKF}
\begin{gathered}
V^{(k)}_{n}=\widehat{V}^{(k)}_{n}-\widehat{C}_{n}H^T(I+H\widehat{C}_nH^T)^{-1}(H \widehat{V}^{(k)}_n-Z^{(k)}_{n}),\\
\widehat{V}^{(k)}_{n}=\Psi_h(V^{(k)}_{n-1})+\zeta^{(k)}_n,\quad \overline{\widehat{V}}_{n}:=\frac{1}{K}\sum^{K}_{k=1}\widehat{V}^{(k)}_n,\quad Z^{(k)}_{n+1}=Z_{n}+\xi^{(k)}_{n},\\
\widehat{C}_{n}:=\frac{1}{K-1}\sum_{k=1}^K(\widehat{V}^{(k)}_{n}-\overline{\widehat{V}}_{n})\otimes (\widehat{V}^{(k)}_{n}-\overline{\widehat{V}}_{n})\;.
\end{gathered}
\end{equation}
From this formulation, it is clear that the augmented process $\{U_n, V^{(1)}_n,\ldots, V^{(K)}_n\}$ is a Markov chain. In the following discussion, we will denote the natural filtration up to time $n$ as $\mathcal{F}_n=\sigma\{U_m, V^{(1)}_m,\ldots, V^{(K)}_m,m\leq n\}$, and denote the conditional expectation with respect to $\mathcal{F}_n$ as $\mathbb{E}_n$.

\subsection{Ensemble square root filters}
\label{sec:ESRFintro}
One drawback of EnKF comes from its usage of artificial noise $\xi_n^{(k)}$, as this introduces unnecessary sampling errors, particularly when the ensemble size is small \cite{evensen04}. The motivation behind the artificial noise is to make the posterior ensemble covariance 
\[
C_{n}:=\frac{1}{K-1}\sum_{k=1}^K(V^{(k)}_{n}-\overline{V}_{n})\otimes (V^{(k)}_{n}-\overline{V}_{n}),\quad  \overline{V}_{n}:=\frac{1}{K}\sum^{K}_{k=1}V^{(k)}_n,
\]
satisfy the covariance update of the standard Kalman filter
\begin{equation}
\label{eqn:postprior}
C_n=\widehat{C}_n-\widehat{C}_nH^T(H^T\widehat{C}_nH+I)^{-1}H\widehat{C}_n,
\end{equation}
when the left hand is averaged over $\xi_n^{(k)}$ \cite{AA99, HWS01, FB07} . 
ESRFs, including the ensemble transform Kalman filter (ETKF) and the ensemble adjustment Kalman filter (EAKF), aim to resolve this issue by manipulating the posterior spreads to ensure that \eqref{eqn:postprior} holds. Both ETKF and EAKF algorithms are described by the following update steps, with the only difference occurring in the assimilation step for the spread. As with EnKF, the initial ensemble $\{V_{0}^{(k)}\}_{k=1}^K$ is (for instance) sampled from the equilibrium distribution of $U_n$. 
\begin{itemize}
\item Forecast step: identical to EnKF, the forecast ensembles at time $n$ is generated from posterior ensembles at time $n-1$:
\[
\widehat{V}_{n}^{(k)}=\Psi_h(V_{n-1}^{(k)})+\zeta^{(k)}_n.
\]
The forecast ensemble covariance $\widehat{C}_n$ is then computed using \eqref{eqn:C}. 
\item Assimilation step for the mean: upon receiving the new observation $Z_n$, the posterior ensemble mean is updated through
\begin{equation}
\label{sys:ESRFmean}
\overline{V}_n=\overline{\widehat{V}_n}-\widehat{C}_{n}H^T(I+H\widehat{C}_nH^T)^{-1}(H \overline{\widehat{V}}_n-Z_{n}),\quad \overline{\widehat{V}}_{n}=\frac{1}{K}\sum^{K}_{k=1}\widehat{V}^{(k)}_n. 
\end{equation}
\item Assimilation step for the spread: The forecast ensemble spread is given by the $d \times K$ matrix 
\[
\widehat{S}_n = [\widehat{V}^{(1)}_n-\overline{\widehat{V}}_n,\dots , \widehat{V}^{(K)}_n-\overline{\widehat{V}}_n  ]\;.
\]
To update the spread, first find a matrix $T_n\in \mathbb{R}^{d\times d}$ (for ETKF) or $A_n\in \mathbb{R}^{K\times K}$ (for EAKF) such that
\begin{equation}
\label{sys:ESRFspread}
\frac{1}{K-1}T_n\widehat{S}_n \otimes  T\widehat{S}_n =\frac{1}{K-1}\widehat{S}_n A_n\otimes  \widehat{S}_n A_n=\widehat{C}_n-\widehat{C}_nH^T(H^T\widehat{C}_nH+I)^{-1}H\widehat{C}_n\;. 
\end{equation}
The posterior spread is updated to $S_n=T_n\widehat{S}_n$ (for ETKF) or $S_n=\widehat{S}_nA_n$ (EAKF), and the ensemble members are updated to
\[
V^{(k)}_n = \overline{V}_n + S_n^{(k)}\;,
\]
where $S_n^{(k)}$ denotes the $k$-th column of the updated spread matrix $S_n$. 
By construction, the posterior covariance $C_n=(K-1)^{-1}S_n^TS_n$  satisfies \eqref{eqn:postprior}. 
\end{itemize}
At this stage it suffices to know that such $A_n$ and $T_n$ exist, their finer properties play no role in the discussion concerning stability. 
%
Their formulation will become important when we want to study ergodicity in Section \ref{sec:ergo}, and a detailed formulation will be given there. Based on our description above,  the augmented process $\{U_n, V^{(1)}_n,\ldots, V^{(K)}_n\}$ is again a Markov chain. As above, we will denote the natural filtration up to time $n$ as $\mathcal{F}_n=\sigma\{U_m, V^{(1)}_m,\ldots, V^{(K)}_m,m\leq n\}$, and denote the conditional expectation with respect to $\mathcal{F}_n$ as $\mathbb{E}_n$.

\subsection{Covariance inflation}

\label{sec:covinfl}
When applying EnKF and ESRF, the forecast ensemble covariance $\widehat{C}_n$ often underestimates the uncertainty in the forecast model. 
An ad hoc solution is to use inflated or modified forecast ensemble covariance in the assimilation step.  We  will discuss three types of such methods in this paper:
\begin{itemize}
\item Additive inflation: replace $\widehat{C}_n$ with $\widehat{C}_n+\lambda I$ for a proper $\lambda>0$;
\item Uniform inflation: replace $\widehat{C}_n$ with $(1+\lambda)\widehat{C}_n$ for a proper $\lambda>0$.
\end{itemize}
It should be noted that additive inflation is only used in EnKF and not in the square root filters since it is not clear how an additive inflation should be applied at the level of the matrix square root. There are other ad hoc ways of modifying ESRF methods with additive inflation, see page 147 of \cite{MH12} for more details. 

\subsection{Energy principles and Lyapunov functions}
Stability for nonlinear systems can be studied through energy principles. That is, certain types of energy are preserved or dissipated {by the dynamics}. In a stochastic setting, this idea is formalized using Lyapunov functions. In this paper,  we say that $\mathcal{E}$ is a \emph{Lyapunov function} for a Markov chain $X_n$  if there exists positive constants $0<\beta<1$ and $K$ such that
\begin{equation}
\label{eqn:timediscreteLyap}
	\mathbb{E}(\mathcal{E}(X_n)|X_{n-1})\leq (1-\beta) \mathcal{E}(X_{n-1})+K
\end{equation}
for all $n \in \mathbb{Z}^+$. For a continuous time Markov process $X_t$ with generator $\mathcal{L}$, the previous relation is replaced by 
\begin{equation}
\label{eqn:generatorcts}
	\mathcal{L}\mathcal{E}(x)\leq -\beta \mathcal{E}(x)+K\;,
\end{equation}
where we only require $\beta,K>0$. As a simple consequence of Gr\"{o}nwall's inequality, the existence of a Lyapunov function implies (respectively) that 
\begin{equation}
\label{eqn:expectation}
\mathbb{E}\mathcal{E}(X_n)\leq (1-\beta)^n\mathbb{E}\mathcal{E}X_0+K/\beta,\quad \text{or}\quad
\mathbb{E}\mathcal{E}(X_t)\leq e^{-\beta t}\mathbb{E}\mathcal{E}X_0+K/\beta.
\end{equation}
In other words, $\mathbb{E}\mathcal{E}(X_n)$ (or respectively $\mathbb{E}\mathcal{E}(X_t)$) can be bounded uniformly in time. In this case, we say that $X_n$ (or $X_t$) is \emph{$\CE$-bounded}. 
\par
When the sub-level sets of $\mathcal{E}$ are compact, we will call $\mathcal{E}$ a \emph{strong Lyapunov function}. This additional requirement implies that an $\CE$-bounded Markov chain revisits a large enough compact set arbitrarily many times and that the distribution $X_n$ forms a tight sequence. Existence of strong Lyapunov functions will be crucial when proving geometric ergodicity. 

In this paper, we will assume the kinetic energy of the process, $\CE (\cdot) =  |\cdot|^2$, is a strong Lyapunov function. Based on our formulation of the random sequence $U_n$, this is equivalent to the following.
\begin{aspt}[Kinetic energy principle]\label{ass:energy}
\label{aspt:kinetic} There exist constants $0 < \beta_h < 1, K_h > 0$, such that
\[
|\Psi_h(u)|^2+ \text{tr}(R_h(u))\leq (1-\beta_h) |u|^2+K_h\;, 
\]
for all $u\in \mathbb{R}^d$ where $R_h$ is the conditional covariance of the system noise defined in \eqref{eqn:Rh}. 
 \end{aspt} 
When the random sequence $U_n$ is generated from discrete time solutions of an SDE $u_t$, this kinetic energy principle can be verified directly by computing the generator \eqref{eqn:generatorcts} or simply checking the drift of the SDE, see \ref{sec:discrete} for more details. Using this convenient argument, we can easily verify that the following examples all satisfy Assumption \ref{aspt:kinetic}. 
\begin{example}[Stochastic turbulence models]
When $U_n$ is a time discretization of the SDE \eqref{sys:flow}, it suffices to require that for certain $\beta, K>0$
\begin{equation}
\label{eqn:cts}
\mathcal{L}|u|^2=\psi(u)\cdot u+\frac{1}{2}\text{tr}(\Sigma\Sigma^T)\leq -\beta |u|^2+ K,
\end{equation}
since then Assumption \ref{aspt:kinetic} would hold with $\beta_h=1-e^{-\beta h}, K_h=Kh$. Relation \eqref{eqn:cts} holds for many stochastic turbulence models, which generally take the form
\[
dU_t=- DU_t dt +B(U_t)dt+f+\Sigma dW_t.
\]
The linear operator $D$ represents damping, so its symmetric part $\frac{1}{2}(D^T+D)$ is positive semidefinite. The nonlinear interaction term $B$ is  energy preserving, with $\langle U_t, B(U_t)\rangle =0$. Then it is easy to verify that \eqref{eqn:cts} holds for $\beta=\frac{1}{4}\lambda_{min}(D^T+D)$ and $K=|f|^2/\beta$. For more information on these models and their application to turbulence, see \cite{MAG05, MH12, MT14cpam}.  
\end{example}

In the following, we present two well known turbulent systems that all satisfy relation \eqref{eqn:cts}. Hence, Assumption \ref{aspt:kinetic} holds for their discrete time formulation.
\begin{example}[Lorenz 96] 
Let $U_t=(u_{1,t},\dots,u_{N,t})$ be an $N\geq 4$ (usually $N=40$) dimensional system, with its dynamics given by
\[
\dot{u}_{i,t}=-u_{i-2,t}u_{i-1,t}+u_{i-1,t}u_{i+1,t}-u_{i,t}+F,
\]
with the periodic boundary condition $u_{t,k} = u_{t,k-N} = u_{t,k+N}$ for all $k$ and where $F$ is a constant forcing. One can easily show that
\[
\dot{|U_t|^2}=-2|U_t|^2+2F\sum_{i=1}^N u_{i,t}
\leq -|U_t|^2+NF^2. 
\] 
\end{example}

\begin{example}[Truncated stochastic Navier-Stokes system] The incompressible stochastic Navier-Stokes  equation on a two dimensional torus can be described through the vorticity field
\[
dv_t=\nu \Delta v_t dt-B(\mathcal{K}v_t, v_t)dt+\sum_{k\in \mathbb{Z}^2}\sigma_k e_kdW_{k,t}.
\]
Here $e_k$ is the Fourier basis for square integrable functions on the torus, so that $e_k(x)=e^{ik\cdot x}$, $\mathcal{K}$ is the linear Biot-Savart integral operator,  mapping $e_k$ to $e_k i k^\bot /|k|^2$, $B$ is the advection effect  $B(u,v):=(u\cdot \nabla)v$ and $(W_{k,t})_{k\in\mathbb{Z}^2}$ is a sequence of independent Wiener processes.  It is well known that for this process the $L^2$-norm is dissipative in time. 

For practical numerical implementation, one needs to truncate the infinite dimensional object $v_t$. For example, one can ignore the high Fourier modes and assume the truncated $v_t$ has the following Fourier decomposition:
\[
\tilde{v}_t=\sum_{k\in I} v_{k,t}e_k+v^*_{k,t}e_{-k}.
\]
Here $*$ denote complex conjugacy, and  
\[
I=\{k=(k_1,k_2)\in \mathbb{Z}^2: |k|\leq N, k\neq \vec{0}, \text{arg}(k_1+k_2 i)\in [0,\pi)\}. 
\]
Note by  formulation, $\tilde{v}_t$ is real valued, and the dynamics of $\tilde{v}_{k,t}$ can be specified by the dynamics of $v_{k,t}$ as follow:
\[
dv_{k,t}=-\nu|k|^2 v_{k,t}dt-\mathbf{P}_k B(\mathcal{K}\tilde{v}_t, \tilde{v}_t)dt +\sigma_k dW_{k,t}.
\]
where $\mathbf{P}_k:v\mapsto \langle v, e_k\rangle $ evaluates the $k$-th Fourier coefficient. Then the full dynamics  $U_t=(v_{k,t})_{k\in I}$ follows a energy principle:
\[
\mathcal{L}|U_t|^2=\sum_{k\in I} \mathcal{L}|v_{k,t}|^2=-2\nu \sum_{k\in I}|k|^2|v_{k,t}|^2+\sum_{k\in I}\sigma_k^2
\leq -2\nu|U_t|^2+\sum_{k\in I}\sigma_k^2. 
\]
Here we used the identity  $\langle v, B(\mathcal{K}v,v)\rangle =0$, so 
\[
\sum_{k\in I} v_{k,t} \mathbf{P}_k B(\mathcal{K}\tilde{v}_t,\tilde{v}_t)=
\langle \sum_{k\in I} v_{k,t} e_k,   B(\mathcal{K}\tilde{v}_t,\tilde{v}_t)\rangle=
\frac{1}{2}\langle \tilde{v}_{t},   B(\mathcal{K}\tilde{v}_t,\tilde{v}_t)\rangle=0. 
\]
\end{example}
In some other nonlinear models, the stability can be demonstrated only after first applying a linear coordinate change. 
%
\begin{example}[Lorenz 63]\label{eg:l63}
\label{exp:lorenz63}
Let $U_t=(x_t,y_t,z_t)$ be a three dimensional system following an ordinary differential equation (ODE):
\[
\frac{d}{dt}x_t=\sigma(y_t-x_t),\quad \frac{d}{dt}y_t=x_t(r-z_t)-y_t,\quad \frac{d}{dt}z_t=x_ty_t-b z_t. 
\]
Then we can define 
\[
\mathcal{E}(U_t)=rx_t^2+\sigma y_t^2+\sigma(z_t-2r)^2,
\]
so using Young's inequality and $\beta:=\min\{2,2\sigma,b\}$,
\begin{align*}
\frac{d}{dt}\mathcal{E}(U_t)&=-2\sigma(rx_t^2+y_t^2+b(z_t-r)^2)+2b\sigma r^2\\
&\leq  -2\sigma(rx_t^2+y_t^2+\tfrac{1}{2}b(z_t-2r)^2)+4b\sigma r^2\\
&\leq -\beta \mathcal{E}(U_t)+4b\sigma r^2.
\end{align*}
One should note here that $|U_t|^2$ does not satisfy relation \eqref{eqn:cts} for all choices of parameters.  
\end{example}
Section \ref{sec:range} will have a detailed discussion of this type of Lyapunov function, where it will be shown that the Lyapunov dissipation relation \eqref{eqn:timediscreteLyap} can be preserved through constant shifts for quadratic functions, but not through linear transformations in general.

%
 
\section{Bounding the observable energy}
\label{sec:stabobs}
In the analysis of ensemble Kalman filters, one natural strategy is obtaining a control over the configuration of the ensemble members. However, such control is very difficult in general, as nonlinear dynamics are known be chaotic and turbulent.  In this section, we will discuss one type of condition which circumvents this problem. Generally speaking, this condition requires that the energy of the observable part, $\CE(\cdot) = |H\cdot|^2$, be a  Lyapunov function for the model $U_n$. In another words, we assume there is a $\beta_h\in (0,1)$ and $K_h>0$  such that 
\[
\mathbb{E}_{n-1} |H U_n|^2\leq (1-\beta_h) |HU_{n-1}|^2+K_h,\quad \text{a.s.}
\] 
This condition can  be formulated in terms of the propagation equation \eqref{sys:discrete}.
\begin{aspt}[Observable energy criterion]\label{ass:energy_obs}
\label{aspt:energydisc} There exists a  $\beta_h\in (0,1)$  and a  $K_h$, such that 
\[
|H\Psi_h(u)|^2+ \text{tr}(HR_h(u)H^T)\leq (1-\beta_h) |Hu|^2+K_h\;,
\]
for all $u \in \mathbb{R}^d$. Here $R_h$ again is the conditional covariance of the system noise in \eqref{eqn:Rh}. 
\end{aspt}
The objective of the current section is to show that this property is inherited by the ensemble and square root Kalman filters, so the observable energy of the ensembles, $\sum_k |H V^{(k)}_n|^2$ has uniformly bounded expectation in time. 

Under the assumption of full observations $H = I_d$, as made in \cite{FB07,KLS14}, it is clear that Assumption \ref{aspt:energydisc} is equivalent to Assumption \ref{aspt:kinetic}, hence the observable energy assumption is quite natural.  
When ${\rm{rank}}(H)=q=d$, the observable energy is actually equivalent to the standard kinetic energy, as  
 \[
 |v|=|(H^TH)^{-1}H^T\cdot Hv|\leq |(H^TH)^{-1}H^T||Hv|,\quad |Hv|\leq |H||v|.
 \]
 However, Assumption \ref{aspt:kinetic} does not in general imply Assumption \ref{aspt:energydisc}. To achieve this implication, one requires that the dissipation in Assumption \ref{aspt:kinetic} be `strong enough'. Section \ref{sec:range} will provide a detailed discussion of when and how Assumption \ref{aspt:energydisc} can be verified.

\subsection{Boundedness of the observable energy for EnKF}

The  advantage of the observable energy $|HU_n|^2$ over the kinetic energy $|U_n|^2$ is that it is preserved in the assimilation step \eqref{sys:EnKF}. To see this, left multiply the assimilation equation by $H$ and rearrange to obtain
\[
HV^{(k)}_{n}
=(I+H\widehat{C}_nH^T)^{-1}H\widehat{V}_n^{(k)}+H\widehat{C}_{n}H^T(I+H\widehat{C}_nH^T)^{-1}Z_n^{(k)}.
\]
The first term on the right can be bounded in terms of $HV^{(k)}_{n-1}$ using Assumption \ref{aspt:energydisc} and the second term can similarly be bounded in terms of $|H U_n|^2$ and an additive constant. This simple observation is the crux of the following result.

\begin{thm}
\label{thm:EnKFobs}
Assume that the signal process $U_n$ satisfies the observable energy criterion, Assumption \ref{aspt:energydisc}, and let $\{V_n^{(k)}\}_{k=1}^K$ be the EnKF ensemble process. Then 
\begin{enumerate}
\item 
There exist constants $D,M>0$ such that 
\begin{equ}\label{e:enkf_diss}
\mathbb{E}_{n-1}(|HV^{(k)}_{n}|^2+M|HU_n|^2)\leq (1-\tfrac{1}{2}\beta_h) (|HV^{(k)}_{n-1}|^2+M|HU_{n-1}|^2)+D
\end{equ}
for each $k =1 \dots K$ and uniformly in $n\geq 1$. In particular, the function 
\[
\mathcal{E}(U, V^{(1)},\ldots, V^{(K)})=\sum_{k=1}^K|HV^{(k)}_{n}|^2+KM|HU_n|^2 
\]
is a Lyapunov function for the Markov chain $(U_n, V^{(1)}_n,\ldots, V^{(K)}_n)$ and hence the signal-ensemble process is $\CE$-bounded. 
\item
When ${\rm{rank}}(H) = q = d$, $\mathcal{E}$ is a strong Lyapunov function and 
\begin{equ}
\BE |U_n|^2+ \sum_{k=1}^K \BE |V^{(k)}_n|^2
\end{equ}
is bounded uniformly in $n\geq 0$. The precise upper bound can be read off directly from \eqref{eqn:expectation} 
\item  Finally, all the claims above hold for any positive semi-definite choice of $\widehat{C}_n$, in particular any covariance inflation scheme from Section \ref{sec:covinfl} satisfies the same relation. 
\end{enumerate}
\end{thm}

\begin{proof}
Left multiply the first equation of \eqref{sys:EnKF} with $H$, 
\begin{align}
\notag
HV^{(k)}_{n}&=H\widehat{V}^{(k)}_{n}-H\widehat{C}_{n}H^T(I+H\widehat{C}_nH^T)^{-1}(H \widehat{V}^{(k)}_n-Z^{(k)}_{n})\\
\label{tmp:EnKFv}
&=(I+H\widehat{C}_nH^T)^{-1}H\widehat{V}_n^{(k)}+H\widehat{C}_{n}H^T(I+H\widehat{C}_nH^T)^{-1}Z_n^{(k)}.
\end{align}
Then based on elementary Lemma \ref{lem:inver} and Young's inequality, Lemma \ref{lem:young}
\[
|HV^{(k)}_{n}|^2\leq (1+\tfrac{1}{2}\beta_h)|H\widehat{V}_n^{(k)}|^2+(1+2\beta_h^{-1})|\widehat{Z}^{(k)}_n|^2. 
\]
Using Assumption \ref{aspt:energydisc} and the conditional independence of $\zeta_n^{(k)}$
\[
\mathbb{E}_{n-1}(|H\widehat{V}_n^{(k)}|^2)=
\mathbb{E}_{n-1}(|H\Psi_h(V_{n-1}^{(k)})+H\zeta_{n-1}^{(k)}|^2)\leq (1-\beta_h)|HV^{(k)}_{n-1}|^2+K_h\;.
\]
Furthermore, by Young's inequality
\[
\mathbb{E}_{n-1}(|\widehat{Z}_n^{(k)}|^2)=
\mathbb{E}_{n-1}(|HU_n+\xi_n+\xi^{(k)}_n|^2)\leq
2\mathbb{E}_{n-1}(|HU_n|^2)+4q\leq 
2|HU_{n-1}|^2+2(K_h+2q).
\]
Combining these inequalities and using $(1-\beta_h)(1+\frac{1}{2}\beta_h)<(1-\frac{1}{2}\beta_h)$ we have
\[
\mathbb{E}_{n-1}|HV^{(k)}_{n}|^2\leq (1-\tfrac{1}{2}\beta_h)|H V_{n-1}^{(k)}|^2+(2+4\beta_h^{-1})|HU_{n-1}|^2+(1+\tfrac{1}{2}\beta_h)K_h+2(1+2\beta_h^{-1})(K_h+2q). 
\]
On the other hand, by Assumption \ref{aspt:energydisc}, for any $M> 0$ we have 
\[
M\mathbb{E}_{n-1}|HU_n|^2\leq M(1-\beta_h) |HU_{n-1}|^2+MK_h. 
\]
Hence, by fixing $M$ such that $\frac{1}{2}\beta_hM>(2+4\beta_h^{-1})$ and adding the previous two inequalities, we see that we can always find a constant $D$ such that
\[
\mathbb{E}_{n-1}(|HV^{(k)}_{n}|^2+M|HU_n|^2)
\leq (1-\tfrac{1}{2}\beta_h)(|HV^{(k)}_{n-1}|^2+M|HU_{n-1}|^2)+D. 
\] 
This completes the proof of the first claim. 
The second claim is simply summing the result of the first claim over all $k$. And when $H$ is of rank $d$, the observable energy $|Hv|^2$ is equivalent to the square energy $|v|^2$
\[
|v|=|(H^TH)^{-1}H^T\cdot Hv|\leq |(H^TH)^{-1}H^T||Hv|.
\]
Finally, notice that we have not used any properties of $\widehat{C}_n$ other than it is positive semi-definite.
\end{proof}

\subsection{Boundedness of the observable energy for ESRF}
The boundedness of ESRF ensembles is not too different from EnKF, since the assimilation step for mean \eqref{sys:ESRFmean} is similar to the assimilation step of EnKF, while the posterior ensemble spread in the observable space can be bounded a.s.
\begin{thm}
\label{thm:EAKF}
Assume that the signal process $U_n$ satisfies the observable energy criterion, Assumption \ref{aspt:energydisc}, and let $\{V_n^{(k)}\}_{k=1}^K$ denote either the EAKF or ETKF ensemble. Then
\begin{enumerate}
\item The observable posterior covariance $HC_nH^T\preceq I_d$ a.s., where 
\[
C_n:=\frac{1}{K}\sum_{k=1}^K (V^{(k)}_n-\overline{V}_n)\otimes (V^{(k)}_n-\overline{V}_n).
\]
\item There exist constants $D, M>0$ such that the ensemble mean  $\overline{V}_n=\frac{1}{K}\sum^K_{k=1}V^{(k)}_n$ satisfies 
\begin{equ}\label{e:ESRF_diss}
\mathbb{E}_{n-1}(|H\overline{V}_n|^2+M|U_n|^2)
\leq (1-\tfrac{1}{2}\beta_h)(|H\overline{V}_{n-1}|^2+M|HU_{n-1}|^2)+D\;,
\end{equ}
for all integers $n \geq 2$. In particular, the function
\[
\mathcal{E}(U, V^{(1)},\ldots, V^{(K)})=\sum_{k=1}^K|HV^{(k)}_{n}|^2+KM|HU_n|^2.
\]
is a Lyapunov function for the signal-ensemble process $(U_n, V^{(1)}_n,\ldots, V^{(K)}_n)$ and hence the process is $\CE$-bounded. 
\item 
When ${\rm{rank}}(H) =d$, $\CE$ is a strong Lyapunov function and 
\[
\BE |U_n|^2+ \sum_{k=1}^K \BE |V^{(k)}_n|^2
\]
is bounded uniformly in $n\geq 0$. The precise bound can be read off directly from \eqref{eqn:expectation}. 
\item Again the claims above hold for any choice of positive semi-definite covariance matrix $\widehat{C}_n$, in particular the uniform covariance inflation scheme in Section \ref{sec:covinfl}.
\end{enumerate}

\end{thm}
\begin{proof}
From the definition of $C_n$ in both ESRF methods, we have that
\begin{align*}
HC_nH^T&=H\widehat{C}_nH^T-H\widehat{C}_nH^T(H^T\widehat{C}_nH+I)^{-1}H\widehat{C}_nH^T\\
&=H\widehat{C}_nH^T(H^T\widehat{C}_nH+I)^{-1}(I+H\widehat{C}_nH)-H\widehat{C}_nH^T(H^T\widehat{C}_nH+I)^{-1}H\widehat{C}_nH^T\\
&=H\widehat{C}_nH^T(H\widehat{C}_nH^T+I)^{-1}. 
\end{align*}
By Lemma \ref{lem:inver}, we clearly have $0\preceq HC_nH^T\preceq I_d$. As a consequence,
\begin{equation}
\label{tmp:sqmean}
\sum_{k=1}^K |HV^{(k)}_n|^2=\text{tr} (HC_n H^T)+K |H\overline{V}_n|^2
\leq K |H\overline{V}_n|^2+d. 
\end{equation}
The inequality \eqref{e:ESRF_diss} follows almost identically to the proof of \eqref{e:enkf_diss}. 
Indeed, the ensemble mean assimilation step \eqref{sys:ESRFmean} implies that  
\begin{align*}
H\overline{V}_n&=H\overline{\widehat{V}_n}-H\widehat{C}_{n}H^T(I+H\widehat{C}_nH^T)^{-1}(H \overline{\widehat{V}}_n-Z_{n})\\
&=(I+H\widehat{C}_nH^T)^{-1}H\overline{\widehat{V}_n}+H\widehat{C}_{n}H^T(I+H\widehat{C}_{n}H^T)^{-1}Z_n. 
\end{align*}
The only difference between this and the proof of \eqref{e:enkf_diss} is that we need to bound $\mathbb{E}_{n-1}|H\overline{\widehat{V}_n}|^2$, but by Jensen's inequality and \eqref{tmp:sqmean} we have
\[
\mathbb{E}_{n-1}|H\overline{\widehat{V}_n}|^2\leq \frac{1}{K}\sum_{k-1}^K\mathbb{E}_{n-1}|H\widehat{V}^{(k)}_n|^2
\leq \frac{1-\beta_h}{K}\sum_{k=1}^K |HV^{(k)}_{n-1}|^2+K_h\leq (1-\beta_h) |H\overline{V}_{n-1}|^2+d+K_h.
\]
Therefore the argument after \eqref{tmp:EnKFv} applies to the process $\overline{V}_n$ verbatim. 

To show that $\mathcal{E}$ is a Lyapunov function, it suffices to apply Young's inequality 
\[
|HV^{(k)}_n|^2\leq (1+\tfrac{1}{4}\beta_h)|H\overline{V}_n|^2+(1+4\beta_h^{-1})|H\overline{V}^{(k)}_n-H\overline{V}_n|^2,
\]
and also see that 
\[
\sum_{k=1}^K|HV^{(k)}_n-H\overline{V}_n|^2=\text{tr}(S_n^TH^THS_n)
=\text{tr}(HS_n S_n^TH^T)=\text{tr}(HC_nH^T)\leq \text{tr}(I_q)\leq q,
\]
where $S_n$ is the posterior spread matrix, which is given by $T_n\Shat_n$ or $\Shat_nA_n$ as in \eqref{sys:ESRFspread}. 
Therefore 
\begin{align*}
\mathbb{E}_{n-1}\mathcal{E}(U_n, V^{(1)}_n,\ldots, V^{(K)}_n)
&\leq (1+\tfrac{1}{4}\beta_h)K\mathbb{E}_{n-1}(|H\overline{V}_n|^2+M|HU_n|^2)+(1+4\beta_h^{-1})q\\
&\leq (1-\tfrac{1}{4}\beta_h)K(|H\overline{V}_{n-1}|^2+M|HU_{n-1}|^2)+KD+(1+4\beta_h^{-1})q\\
&\leq (1-\tfrac{1}{4}\beta_h)(\sum_{k=1}^K|HV^{(k)}_{n-1}|^2+KM|HU_{n-1}|^2)+KD+(1+4\beta_h^{-1})q\\
&=(1-\tfrac{1}{4}\beta_h)\mathcal{E}(U_{n-1}, V^{(1)}_{n-1},\ldots, V^{(K)}_{n-1})+KD+(1+4\beta_h^{-1})q.
\end{align*}
Here we applied Jensen's inequality inequality in the penultimate step. The proofs for the second two claims are identical to 
Theorem \ref{thm:EnKFobs}. 
\end{proof}

\section{Validity of the observable energy criterion}
\label{sec:range}
In Section \ref{sec:stabobs}, we have established a series of uniform boundeness results based on the observable energy criterion, Assumption \ref{aspt:energydisc}. This criterion is different from the usual energy principle for dynamical systems, Assumption \ref{aspt:kinetic}, although they share similar formulations. In this section, we will demonstrate a few sufficient conditions that lead to Assumption \ref{aspt:kinetic} when the observation is of full rank, and discuss a few concrete examples where the observation is rank deficient and Assumption \ref{aspt:kinetic} still holds. 

\subsection{Full rank observation}
As stated earlier, even when $H$ is of full rank, Assumption \ref{ass:energy} does not necessarily imply Assumption \ref{ass:energy_obs}. Indeed, even though the norms $|\cdot|^2$ and $|H\cdot|^2$ are equivalent, the constants of proportionality may preclude the dissipation relation of Assumption \ref{ass:energy_obs}. In a separate work \cite{KMT15}, the authors have constructed a concrete nonlinear system which satisfies the kinetic energy principle, Assumption \ref{aspt:kinetic}, but not  the observable energy principle, Assumption \ref{aspt:energydisc}, and which exhibits catastrophic filter divergence \cite{MH08, GM13} with large probability. This indicates the importance of verifying the observable energy criterion over the typical energy criterion. 
\par
If the condition number of the matrix $H$ is small enough, a kinetic energy principle implies an observable energy principle. To be specific, define the condition number 
%
%
%
\[
	\mathcal{C}_H:=\max \{|Hu||v|: |u|=1, |Hv|=1\},
\]
which must be finite since $H$ is of full rank. Then we have the following. 
\begin{thm} If the kinetic energy principle, Assumption \ref{ass:energy}, holds, then 
\label{thm:strong}
\[
|H\Psi_h(u)|^2+\text{tr}(HR_h(u)H^T)\leq (1-\beta_h)\mathcal{C}^2_H |Hu|^2+|H|^2 K_h
\]
where $\mathcal{C}_H$ is the condition number of matrix $H$. In particular, if $(1-\beta_h)\mathcal{C}^2_H<1$, then the observable energy criterion, Assumption \ref{aspt:energydisc}, also holds; therefore the average square norms of the ensemble members for EnKF and ESQF are bounded uniformly in time as a consequence of Theorems \ref{thm:EnKFobs} and \ref{thm:EAKF}. 
\end{thm}
\begin{proof}
By  Assumption \ref{aspt:kinetic} 
\begin{align*}
|H\Psi_h(u)|^2+\text{tr}(HR_h(u)H^T)&\leq |H|^2[|\Psi_h(u)|^2+\text{tr}(R_h(u))]\\
&\leq (1-\beta_h)|H|^2 |u|^2+ |H|^2 K_h,\\
&\leq (1-\beta_h)\mathcal{C}^2_H |Hu|^2+|H|^2 K_h. 
\end{align*}
\end{proof}
Another situation in which a dissipation criterion implies Assumption \ref{ass:energy_obs} is when the dissipation is of a higher polynomial order. This is an immediate consequence of Theorem \ref{thm:strong}. 
%
\begin{cor} \label{cor:higher_diss} 
Suppose that the stochastic system $U_{n}$ is generated through the solution of the SDE
\[
du_t=\psi(u_t)dt+\Sigma dW_t,
\]
as $U_{n}=u_{nh}$, while for some $\delta,K>0$
\[
\langle u, \psi(u)\rangle \leq -\delta |u|^{1+\delta}+K.
\]  Then the kinetic energy principle, Assumption \ref{aspt:kinetic}, holds with any fixed $\beta_h\in (0,1)$. By Theorem \ref{thm:strong}, if $\text{rank}(H) = d$, then the expected square norms of the ensemble members for EnKF and ESQF are bounded uniformly in time.
 \end{cor}
\begin{proof}
Since
\[
\mathcal{L} |u_t|^2=2\langle u_t, \psi(u_t)\rangle+\text{tr}(\Sigma\Sigma^T)
\leq -2\delta|u|^{1+\delta}+ 2K+\text{tr}(\Sigma\Sigma^T).
\]
By H\"{o}lder's inequality, for any $\alpha>0$, there is a constant $D$ such that
\[
\mathcal{L}|u_t|^2\leq -\alpha |u_t|^2+D.
\]
By Gr\"{o}nwall's inequality and Dynkin's formula we obtain that
\[
\mathbb{E} |U_1|^2=\mathbb{E} |u_h|^2\leq e^{-\alpha h} |u_0|^2+D\alpha^{-1}=e^{-\alpha h} |U_0|^2+D\alpha^{-1}. 
\]
So letting $\alpha=-\ln(1-\beta_h) h^{-1}$ we conclude our proof. 
\end{proof} 
The higher order dissipation described in Corollary \ref{cor:higher_diss} has been used in many stochastic turbulence models to obtain better stability, like  the canonical scalar model with cubic nonlinearity \cite{MFC09, MB12} and  the conceptual dynamical model for turbulence in \cite{ML14}. Moreover, this corollary indicates that {in the full rank case} EnKF can be stabilized if we are willing to filter with a model error that stabilizes the system. For example instead of running the EnKF with vector field $\psi$ of Lorenz 63 or 96, we can run EnKF with the altered system $\tilde{\psi}=\psi-\lambda|u|u$, with any strictly positive $\lambda$, then by the above results the filter will have bounded observable energy on an infinite time horizon.  
\par
We now address the situation where an energy principle holds for a linearly translated version of the typical energy. That is
\[
\mathbb{E}_{n-1}|HU_n-u^*|^2\leq \beta |HU_{n-1}-u^*|^2+K 
\]
for some fixed $u^* \in \reals^q$. Then by Young's inequality  Lemma \ref{lem:young}
\begin{align}
\notag
\mathbb{E}_{n-1}|HU_n|^2&\leq 
(1+\epsilon)\mathbb{E}_{n-1}|HU_n-u^*|^2+(1+\epsilon^{-1})|u^*|^2\\
\notag
&\leq (1+\epsilon)\{\beta |HU_{n-1}-u^*|^2+K\}+(1+\epsilon^{-1})|u^*|^2\\
\label{tmp:linear}
&\leq (1+\epsilon)^2 \beta |HU_{n-1}|^2+ (1+\epsilon^{-1})^2 (K+2|u^*|^2). 
\end{align}
Therefore our framework can be applied to the Lorenz 63 system, Example \ref{exp:lorenz63}, as long as $H/\text{diag}\{r,\sigma,\sigma\}$ satisfies the condition number requirement of Theorem \ref{thm:strong}. 
\par

\subsection{Observation with rank deficiency}
In the case where $H$ is rank deficient we do not expect the observable energy criterion to hold for broad classes of models. This constraint is in a sense not surprising, since EnKF and ESQF with sparse observations can potentially diverge to machine infinity, in a well documented phenomena known as catastrophic filter divergence \cite{MH08, MH12, GM13}. Nevertheless, under certain scenarios one can verify Assumption \ref{aspt:energydisc} through explicit calculation. We will now verify the observable energy criterion for two concrete examples. 

\begin{example}
Consider a stable linear dynamical system given by $U_n=AU_{n-1}+\zeta_n$, where $A$ produces a contraction,
\[
|Au|\leq (1-\beta)|u|\quad \text{ for all }u,
\]
with a $\beta\in (0,1)$. For simplicity, we assume that $\zeta_n$ are i.i.d. random variables with mean $0$ and covariance $R$. Suppose that $H$ commutes with $A$, $AH=HA$, then  
 \begin{align*}
\mathbb{E}(|HU_n|^2|U_{n-1}=u)&=|HAu|^2+\mathbb{E}|H\zeta_n|^2\\
&=|AHu|^2+\text{tr}(HRH^T)\\
&\leq (1-\beta)^2|Hu|^2+\text{tr}(HRH^T). 
 \end{align*}
Hence Assumption \ref{aspt:energydisc} holds. 
\end{example}

\begin{example}
Recall  the Lorenz 63 model, Example \ref{exp:lorenz63}, where $U_n$ is given by $u_{nh}=(x_{nh}, y_{nh},z_{nh})$, where the dynamics of $u_t$ is specified by 
\[
\frac{d}{dt}x_t=\sigma(y_t-x_t),\quad \frac{d}{dt}y_t=x_t(r-z_t)-y_t,\quad \frac{d}{dt}z_t=x_ty_t-b z_t. 
\]
Suppose that we have direct noisy observations of the latter two coordinates, that is  $H={\text{diag}} \{0,1,1\}$. We can define the linearly translated observable energy as $\mathcal{E}_H(u_t)=y_t^2+(z_t-r)^2$. Direct computation yields
\[
\frac{d}{dt}\mathcal{E}_H(u_t)=-2y_t^2-2bz^2_t+2br z_t\leq -\gamma \mathcal{E}_H(u_t)+r^2,
\]
where $\gamma=\min\{2,b\}$. Therefore by Gr\"{o}nwall's inequality, 
\[\mathcal{E}_H(u_{t+h})\leq e^{-\gamma h}\mathcal{E}_H(u_{t})+hr^2.
\] 
Following our manipulation for linearly translated energy, \eqref{tmp:linear}, we can conclude that Assumption \ref{aspt:energydisc} holds by taking a sufficiently small $\epsilon\leq 1$ in the following application of Lemma \ref{lem:young}
\begin{align*}
|HU_n|^2=y^2_{nh}+z^2_{nh}&\leq (1+\epsilon)\mathcal{E}_H(u_{nh})+(1+\epsilon^{-1})r^2\\
&\leq (1+\epsilon)e^{-\gamma t}\mathcal{E}_H(u_{(n-1)h})+(1+\epsilon)hr^2+(1+\epsilon^{-1})r^2\\
&\leq (1+\epsilon)^2 e^{-\gamma t}(y^2_{(n-1)h}+z^2_{(n-1)h})+3(1+\epsilon^{-1})r^2+(1+\epsilon)hr^2\\
&=(1+\epsilon)^2e^{-\gamma t}|HU_{n-1}|^2+3(1+\epsilon^{-1})r^2+(1+\epsilon)hr^2. 
\end{align*}
\end{example}

\section{Geometric ergodicity of the ensemble based Kalman filters}
\label{sec:ergo}
The objective of this section is to verify geometric ergodicity for the signal-ensemble process. In particular, if $P$ denotes the Markov transition kernel for the signal ensemble process, then we will show that there exists a constant $\gamma \in (0,1)$ such that
\begin{equ}
\norm{P^n \mu - P^n \nu}_{TV} \leq C_{\mu,\nu}\gamma^n\;,
\end{equ}
where $\mu,\nu$ are two arbitrary initial probability distributions, $C_{\mu,\nu}$ is a time uniform constant that depends on $\mu,\nu$, and $\norm{\cdot}_{TV}$ denotes the total variation norm. Furthermore the nonlinear filter has a unique invariant attracting state, the analogue of the asymptotic filter for linear Kalman filters \cite{Jaz72, MH12}. Hence, geometric ergodicity implies that discrepancies in the initialization of the ensemble filters, which is usually inevitable in practice, will dissipate exponentially with time.

To prove the ergodicity of the signal-ensemble process, we invoke a standard result of Markov chain theory \cite{MT93}. Here we use a simple adaptation of the form given in \cite[Theorem 2.3]{MS02}.
\begin{thm}
\label{thm:MS02}
Let $X_n$ be a Markov chain in a space $E$ such that
\begin{enumerate}
\item There is a strong Lyapunov function $\mathcal{E}:E\mapsto \mathbb{R}^+$ for the Markov process $X_n$ 
\item For any fixed $M>0$,  the compact set $C=\{x:\mathcal{E}(x)\leq M\}$ satisfies the minorization assumption. That is, there is a probability measure $\nu$ with $\nu(C)=1$, and a $\eta>0$ such that for any given set $A$
\[
\mathbb{P}(X_n\in A|X_{n-1}=x)\geq \eta\nu(A)
\]
for all $x\in C$. 
\end{enumerate}
Then there is a unique invariant measure $\pi$ and constants $r \in (0,1), \kappa>0$ such that
\[
\|\mathbb{P}^\mu(X_n\in \,\cdot\,)-\pi\|_{TV}\leq \kappa r^n \bigg(1+\int \CE(x) \mu(dx)\bigg). 
\]
\end{thm}

In the following we discuss the conditions we require in order to apply Theorem \ref{thm:MS02}. In the previous sections, we established the existence of Lyapunov functions for signal-ensemble processes. In particular, when $H$ is of full rank,  Theorems \ref{thm:EnKFobs}, \ref{thm:EAKF} and \ref{thm:strong} provide simple criteria which guarantee that the filtering process satisfies the following assumption
\begin{aspt} [Existence of a strong Lyapunov function]
\label{aspt:Lyapunov}
 \label{ass:lyap}
There is a function $\mathcal{E}:\reals^d \times \reals^{d\times K} \to  \mathbb{R}^+$ with compact sublevel sets and constants $K_h, \beta_h>0$ such that
\[
\mathbb{E}_{n-1}\mathcal{E}(U_n,V^{(1)}_n,\ldots,V^{(K)}_n)\leq (1-\beta_h)\mathcal{E}(U_{n-1},V^{(1)}_{n-1},\ldots,V^{(K)}_{n-1})+K_h\;.
\]
\end{aspt}
Although Assumption \ref{aspt:Lyapunov} may be difficult to hold in general scenarios, the authors have found an adaptive covariance inflation scheme that always guarantees Assumption \ref{aspt:Lyapunov}, which will be reported in a separate paper \cite{TMK15}. This assumption provides the first hypothesis of Theorem \ref{thm:MS02}.

In order to verify the minorization condition of Theorem \ref{thm:MS02}, we need to assume there is a density for the noise $\zeta_n$ appearing in the time discrete model \eqref{sys:discrete}\;.
\begin{aspt}[Nondegenerate system noise]
\label{aspt:density}
For any $M_1, M_2>0$, there is a constant $\alpha>0$ such that 
\[
\mathbb{P}(\zeta_n\in \cdot \; |U_{n-1}=u)\geq \alpha\lambda_{M_2}(\cdot)
\]
for all $|u|\leq M_1$, where $\lambda_{M_2}(dx)$ is the Lebesgue measure of $\mathbb{R}^d$ restricted to $\{u: |u|\leq M_2\}$.
\end{aspt}
Assumption \ref{aspt:density} holds for many practical examples. When $U_n$ is produced by time discretization of an SDE \eqref{sys:flow}, it suffices to require $\Sigma$ being nonsingular, please see  \ref{sec:discrete} for a detailed discussion. In other situations where $U_n$ is produced genuinely as a random sequence, $\zeta_n$ is usually a sequence of random Gaussian variables, which Assumption \ref{aspt:density} also satisfies.

\subsection{Controllability}
\par
In this subsection, we establish a framework which verifies the minorization condition using Assumption \ref{aspt:density} and controllability of the Kalman update map. The signal-ensemble process $X_n := (U_n , V_n^{(1)},\dots,V_n^{(K)})$ is a Markov chain taking values in $\CX = \reals^{d }\times \reals^{d \times K}$. For all three ensemble filters, the evolution of $X_n$ is described by the composition of two maps. The first is a random map from $\CX$ to a signal-forecast-observation space $\CY$, described by a Markov kernel $\Phi : \CX \times \CB(\CY) \to [0,1]$. The second is a deterministic map $\Gamma : \CY \to \CX$, which combines the forecast with the observed data to produce the updated posterior ensemble. The details of these maps, as well as the definition of the intermediate space $\CY$, differs between EnKF, ETKF and EAKF. 
\par
For EnKF, the intermediate space is $\CY := \reals^d \times \reals^{d\times K } \times \reals^{q \times K}$ and the random mapping is 
\begin{equ}
(U_{n-1} , V_{n-1}^{(1)},\dots,V_{n-1}^{(K)} ) \mapsto Y_n :=  (U_n, \Vhat_n^{(1)},\dots,\Vhat_n^{(K)},  Z_n^{(1)},\dots,Z_n^{(K)} )\;.
\end{equ}
The deterministic map $\Gamma$ is given by 
\begin{equ}
\Gamma (U_n , \Vhat_n^{(1)},\dots,\Vhat_n^{(K)}, Z_n^{(1)},\dots,Z_n^{(K)}) = (U_n , \Gamma^{(1)},\dots,\Gamma^{(K)})
\end{equ}
where 
\begin{equs}\label{e:Gamma_enkf}
\Gamma^{(k)} &= \widehat{V}^{(k)}-\widehat{C}H^T(I+H\widehat{C}H^T)^{-1}(H \widehat{V}^{(k)}-Z^{(k)}) \\
\widehat{C} &= \frac{1}{K-1}\sum_{k= 1}^{K}(\widehat{V}^{(k)}-\overline{\widehat{V}})\otimes (\widehat{V}^{(k)}-\overline{\widehat{V}}) \quad \Vhatbar = \frac{1}{K}\sum_{k=1}^K \Vhat^{(k)}\;.
\end{equs}
The corresponding formulas for ETKF and EAKF will be given in Sections \ref{s:erg_etkf} and \ref{s:erg_eakf} respectively. 
\par
Given this formuation, it suffices to show that the push-forward kernel $\Gamma^* \Phi (x,\cdot)  = \Phi(x,\Gamma^{-1}(\cdot))$ satisfies the minorization condition. It is easy to see that, given the assumptions on the noise, the kernel $\Phi(x,\cdot)$ has a density with respect to Lebesgue measure, so we simply need to show that the pushforward inherits the density property from $\Phi$. To achieve this, we use the following simple fact.  

\begin{lem}
\label{lem:twosteps}
Let $\Phi$ be a Markov transition kernel from $\reals^{n} \to \reals^{n} \times \reals^m$ with a Lebesgue density $p(x,y) = p(x,(y_1,y_2))$ and let $\Gamma : \reals^n \times \reals^m \to \reals^n$. 
%
Given a compact set $C$, suppose that there is a point $y^*=(y_1^*, y_2^*) \in \reals^n \times \reals^m$ and $\beta > 0 $ such that 
\begin{enumerate}
\item For all $x\in C$, the density function $p(x,y)>\beta$ for $y$ around $y^*$\;,
\item $\Gamma$ is $C^1$ in a neighborhood of  $y^*$ and $\det (\CD_{y_1} \Gamma)|_{y^*}>0$\;. 
\end{enumerate}
Then there is a $\delta>0$ and a neighborhood $O_1$ of $\Gamma(y^*)$ such that for all $x\in C$ 
\[
\Gamma^* \Phi(x, \cdot )\geq \delta\lambda_{O_1}(\cdot)
\]
where $\lambda_{O_1}$ is the Lebesgue measure restricted to the set $O_1$. In other words, the minorization condition holds for the transition kernel $\Gamma^*\Phi$. 
\end{lem}
\begin{proof}
By continuity and compactness, we can find  a small neighborhood $O$ of $y^*$, a neighborhood $O'$ of $(\Gamma(y^*),y_2^*)$, and a small positive number $\epsilon>0$ such that for all $x\in C, (y_1,y_2)\in O$
\[
p(x,(y_1,y_2))>\epsilon, \quad \epsilon^{-1}>|\det (\CD_{y_1} \Gamma(y_1,y_2))|>\epsilon\;.
\]
such that $\Gamma$ is a $C^1$ diffeomorphism from $O$ to $O'$. Let $D_0 = \{ (y_1,y_2) : \Gamma (y_1,y_2 ) \in A , y_2 \in B , (y_1,y_2) \in O \}$ then, using the change of variables $(y_1,y_2) \mapsto (\Gamma(y_1,y_2),y_2)$, we have by the change of variables formula
\begin{equs}
\Phi(x,D_0) &:= \int \unit_{A} (\Gamma(y_1,y_2)) \unit_{B}(y_1,y_2) p(x,y_1,y_2) dy_1 dy_2\\ &= \int\unit_{(A \times B)}(z,y_2) \unit_{O'}(z,y_2) q(x, (z,y_2)) dz dy_2
\end{equs} 
where 
\[
q(x,(z,y_2))=p(x,y_1,y_2) \bigg|\det\begin{bmatrix} \CD_{y_1} z & \CD_{y_2} z\\
\CD_{y_1} y_2 &\CD_{y_2} y_2
\end{bmatrix}\bigg|^{-1}=p(x,y_1,y_2) |\det (\CD_{y_1} \Gamma(y_1,y_2))|^{-1}.
\]
with $z = \Gamma (y_1,y_2)$. 
%
By construction, $q$ is strictly above $\epsilon^2$ for $x\in C$ and $(z,y_2)\in O'$. Pick neighborhood $O_1$ of $z^*=\Gamma(y^*)$ and $O_2$ of $y_2$ such that $O_1\times O_2\subset O'$. Then, since the set $\Gamma^{-1}(A) \cap O$ is of the form $D_0$ (with $B = \reals^m$) we can apply the above change of variables formula to obtain
\begin{equs}
\Phi(x,\Gamma^{-1} (A)) \geq \Phi (x,\Gamma^{-1}(A)\cap O) &= \int \unit_{A}(z) \unit_{O'}(z,y_2) q(x,z,y_2) dz dy_2 \\
& \geq \int_{O_1 \times O_2} \unit_{A}(z)  q(x,z,y_2) dz dy_2 \\
& \geq \epsilon^2 \lambda(O_1) \lambda(O_2 \cap A)\;.
\end{equs}
%
Then taking $\delta=\epsilon^2 \lambda(O_2)$ satisfies our requirement. 
\end{proof}

\subsection{Ergodicity for the EnKF}
In the application of Lemma \ref{lem:twosteps} to the signal-ensemble process, we will use the variables
\begin{equs}
x &= (U_{n-1} , V_{n-1}^{(1)},\dots,V_{n-1}^{(K)} )\\  y_1 &= (U_n, \Vhat_n^{(1)},\dots,\Vhat_n^{(K)}) \\ y_2 &= (Z_n, Z_n^{(1)},\dots,Z_n^{(K)} )
\end{equs}
%
The choice of the intermediate point $(y_1^*,y_2^*)$ can be quite delicate. Although the non degeneracy of the Jacobian should in principal be verifiable for general $\Gamma$, a well chosen intermediate point can simplify the computation significantly.

With Theorem \ref{thm:MS02} and Lemma \ref{lem:twosteps}, the verification of EnKF becomes rather straight forward.
\begin{thm}
\label{thm:EnKFergo} If the unfiltered signal process $U_n$ has an kinetic energy principle with nondegenerate system noise, and the EnKF signal-ensemble process has a strong Lyapunov function, in other words Assumptions \ref{aspt:kinetic}, \ref{aspt:density} and \ref{aspt:Lyapunov} hold,  then the EnKF signal-ensemble process is geometrically ergodic in total variation distance.
\end{thm}
\begin{proof}
Fix any $M_1>0$, we apply Lemma \ref{lem:twosteps} to the compact set
\[
C=\bigg\{(u,v^{(1)},\ldots, v^{(K)}): |u|^2+\sum_{k=1}^K|v^{(k)}|^2\leq M_1\bigg\}.
\]
Pick the intermediate point $y^*$ with all its components at the origin. It is easy to see that the first condition of Lemma \ref{lem:twosteps} holds. Indeed the random variable $(\zeta_n, \zeta_n^{(k)},\xi_n,\xi_n^{(k)})$ satisfies the density condition by assumption and $(y_1,y_2)$ is obtained from this random variable via an onto linear transformation, hence $(y_1,y_2)$ inherits the density condition.

It is also elementary to verify the differentiability and nondegeneracy of $\Gamma$ at $y^*$, where $\Gamma$ is defined by \eqref{e:Gamma_enkf}. Indeed, using the formula for gradients of inverse matrices $\CD L^{-1} = -L^{-1}\CD L L^{-1}$, it is clear that $\Gamma$ is a polynomial combination of several continuously differentiable functions and in particular must be $C^1$ near $y^*$. To prove non-degeneracy, notice that both $\Chat$ and $\CD_{y_1} \Chat$ vanish at $y^*$. Using this fact, a simple calculation yields 
\begin{equ}
\CD_{y_1}\Gamma |_{y^*} = I\;,
\end{equ} 
which proves non-degeneracy and hence the Theorem follows from Lemma \ref{lem:twosteps} and Theorem \ref{thm:MS02}.  
%
%

 \end{proof}
\subsection{Ergodicity of ETKF}\label{s:erg_etkf}

For both ETKF and EAKF, the intermediate space is slightly different since there are no longer perturbed observations. In particular we have $\CY : = \reals^d \times \reals^{d \times K} \times \reals^q$ and the Markov kernel $\Phi : \CX \times \CB(\CY) \to [0,1]$ is described by 
\begin{equ}
(U_{n-1} , V_{n-1}^{(1)},\dots,V_{n-1}^{(K)}) \mapsto (U_n, \Vhat_{n}^{(1)},\dots,\Vhat_{n}^{(K)},Z_n)\;.
\end{equ}
The deterministic step is given by the map  $\Gamma(U , V , Z ) = (U,\Gamma^{(1)},\dots,\Gamma^{(K)})$ where
\begin{equ}\label{e:Gamma_etkf}
\Gamma^{(k)} = \Vhatbar - \Chat H^T (I + H \Chat H^T)^{-1} (H \Vhatbar - Z)  + S^{(k)}
\end{equ}
where $\Vbar = \frac{1}{K}\sum_{k=1}^K \Vhat^{(K)}$ and $\Chat = \frac{1}{K-1}\sum_{k=1}^K (\Vhat^{(K)} - \Vbar) \otimes (\Vhat^{(K)} - \Vbar)  $ and $S^{(k)}$ is the $k$-th column of the updated spread matrix $S = \Shat T(\Shat)$ where $\Shat$ is the forecast spread matrix $\Shat = (\Vhat^{(1)} - \Vhatbar , \dots ,\Vhat^{(K)} - \Vhatbar)$ and $T(\Shat)$ is the transform matrix. We have not yet defined the transform matrix $T(\Shat)$, other than to require that it yields the covariance condition \eqref{sys:ESRFspread}. One reasonable choice for the transform matrix, which we will adopt, is the matrix square root
\begin{equ}
\label{sys:ETKFspread}
T(\Shat)=\big(I_K+(K-1)^{-1}\widehat{S}^TH^TH\widehat{S}\big)^{-1/2}=\big(I_K-(K-1)^{-1}\widehat{S}^TH^T(I+H\widehat{C}H^T)^{-1}H\widehat{S}^T\big)^{1/2}\;.
\end{equ} 
Note that the square root is well defined and unique since the argument is symmetric and positive semi-definite. We will now apply Lemma \ref{lem:twosteps} with 
\begin{equs}
x = (U_{n-1}, V_{n-1}^{(1)},\dots,V_{n-1}^{(K)})\quad
y_1  = (U_n, \Vhat_n^{(1)},\dots,\Vhat_n^{(K)})\quad
y_2  = Z_n\;.
\end{equs}
and the intermediate point $y^* = (y_1^*,y_2^*) = (0,0)$.

\begin{thm}
\label{thm:ETKFergo}
If the unfiltered signal process $U_n$ has an kinetic energy principle with nondegenerate system noise, and the ETKF signal-ensemble process has a strong Lyapunov function, in other words Assumptions \ref{aspt:kinetic}, \ref{aspt:density} and \ref{aspt:Lyapunov} hold,  then the ETKF signal-ensemble process is geometrically ergodic.
\end{thm}

\begin{proof}
Fix any number $M_1>0$ and define the compact set
\[
C=\bigg\{(u,v^{(1)},\ldots, v^{(k)}): |u|^2+\sum_{k=1}^K|v^{(k)}|^2\leq M_1\bigg\}.
\]
As with Theorem \ref{thm:EnKFergo}, we pick the intermediate point $y^*$ to be the origin. Showing that $\Phi$ satisfies the first condition of Lemma \ref{lem:twosteps} is identical to Theorem \ref{thm:EnKFergo}, hence it suffices to show differentiability and non-degeneracy.  
\par
By Lemma \ref{lem:matrixhalf}, the transform matrix $T(\Shat)$ is continuously differentiable and hence it follows trivially that $\Gamma$ is $C^1$ near $y^*$. For the non-degeneracy condition, we have that
\begin{equ}\label{e:Dy1_eakf}
\CD_{y_1} \Gamma^{(k)} = \CD_{y_1} \Vhatbar - \CD_{y_1} \big(\Chat H^T (I + H\Chat H^T)^{-1} (H\Vhatbar - Z) \big) + \CD_{y_1} S^{(k)}\;.
\end{equ}
Precisely as in Theorem \ref{thm:EnKFergo}, the second term on the right hand side vanishes at $y^*$. For the third term, we have by Leibniz rule
\begin{equ}
\CD_{y_1} S = \CD_{y_1} \Shat T(\Shat) + \Shat \CD_{y_1} T(\Shat)\;.  
\end{equ}
But by the definition of $T(\Shat)$, it is clear that $\CD_{y_1} (T(\Shat) T(\Shat))$ vanishes as $y^*$ and that $T(\Shat) = I$ at $y^*$. Hence we have
\begin{equ}
0 = \CD_{y_1} (T(\Shat) T(\Shat))|_{y^*} = 2 \CD_{y_1}T(\Shat)|_{y^*}\;,
\end{equ}
thus $\CD_{y_1} S|_{y^*} = \CD_{y_1} \Shat |_{y^*} $. Returning to \eqref{e:Dy1_eakf}, we see that
\begin{equ}
\CD_{y_1} \Gamma^{(k)}|_{y^*} = \CD_{y_1}\Vhatbar |_{y^*} + \CD_{y_1} \Shat^{(k)} |_{y^*} =   \CD_{y_1} \Vhat^{(k)} |_{y^*}\;.
\end{equ}
It follows that $\CD_{y_1}\Gamma |_{y^*} = I$ and as in Theorem \ref{thm:EnKFergo}, this completes the proof. 

\end{proof}

\begin{rem}
ETKF has different formulations to \eqref{sys:ETKFspread}, where $T(\widehat{S})$ is obtained by multiplying the original formula by a rotation matrix on the right. The same principles apply to these ETKF, as long as the rotation matrix is differentiable around the intermediate point. However, the origin may not be good choice of intermediate point, since the rotation map can be singular to perturbations around the origin. In this case, one must use the methods in the following section for EAKF, where we deal with the same issue. 
\end{rem}

%

\subsection{Ergodicity of EAKF}\label{s:erg_eakf} 

%

In this section, we apply the same strategy as EnKF and ETKF to obtain geometric ergodicity for EAKF. For EAKF, the intermediate space $\CY$ and the Markov kernel $\Phi$ are identical to those of ETKF. The deterministic map $\Gamma$ is still defined by \eqref{e:Gamma_etkf}, but now the spread matrix $S$ is defined by $S = A(\Shat)\Shat$ where $A(\Shat)$ is an adjustment matrix. As with ETKF, the adjustment matrix can be any matrix that ensures the covariance condition \eqref{sys:ESRFspread}. In the next section, we will discuss how to construct such an adjustment matrix.

\subsubsection{Detailed formulation of EAKF}
\label{sec:formulationEAKF}

We adopt the formulation of EAKF from \cite{And01, MH12}, described by the following steps. 
\begin{enumerate}
\item Compute the SVD decomposition of $\widehat{S}$, denoted by $\widehat{S}=Q \Lambda R$. When there is rank deficiency, i.e. $\Lambda$ is not square or not invertible, we can further decompose $Q$ and $R$ into parts corresponding to null and complementary subspaces: 
\begin{equation}
\label{eqn:SVD}
\widehat{S}=\begin{bmatrix}Q_1 &Q_2\end{bmatrix} \begin{bmatrix}\Lambda_1 &0\\
0 &0\end{bmatrix} \begin{bmatrix} R_1\\ R_2 \end{bmatrix}=Q_1 \Lambda_1 R_1
\end{equation}
where $\Lambda_1$ is a square diagonal invertible $k \times k$ matrix with $k\leq \min(d,K)$. Without loss of generality, we assume $\Lambda_1$ has its diagonal entries descending.
\item Let $G^TDG$ be the eigenvalue decomposition of $(K-1)^{-1}\Lambda^T Q^TH^THQ\Lambda$ where $D$ is positive semi-definite with decreasing diagonal entries.  As in the first step, in the rank deficient case we can write
\[
G_1^TD_1G_1=(K-1)^{-1}\Lambda_1 Q^T_1H^THQ_1\Lambda_1
\]
where $G_1, D_1$ are $k\times k$ matrices. 
\item In the general (rank deficient) case, the assimilation matrix is given by
\begin{equation}
\label{sys:EAKFspreaddeficient}
A(\Shat)=Q_1 \Lambda_1 G_1^T(I+D_1)^{-1/2}\Lambda^{-1}_1Q_1^T\;
\end{equation}
or equivalently updating the spread matrix to be
\[
S=Q_1\Lambda_1 G_1^T(I+D_1)^{-1/2}R_1. 
\]
In the full rank case this reduces to the more well known formulation
\begin{equation}
\label{sys:EAKFspread}
A(\Shat)=Q\Lambda G^T(I+D)^{-1/2} \Lambda^{\dagger}Q^T\;.
\end{equation}

\end{enumerate}

In existing EAKF literature, the rank deficient case is rarely examined except as a footnote in \cite{TAB03}. However, rank deficiency is  unavoidable when the ensemble size $K$ is less than the model dimension $d$. As a matter of fact, in these degenerate scenarios, the choice of the eigen-basis is very subtle, and simply following the classical formulation \eqref{sys:EAKFspread} with different choice of $G$ could end up with inaccurate posterior covariance.  The assimilation matrix formulation \eqref{sys:EAKFspreaddeficient} we used here guarantees that the posterior covariance follows the Kalman update relation \eqref{eqn:postprior}. In practice, direct application of \eqref{sys:EAKFspread} in most mathematical programs with default settings will generate the same update rule.  A more detailed discussion of these issues are presented in \ref{sec:EAKFpatho}.



Since the above EAKF formulation relies on a singular value decomposition, the choice of singular vectors and hence the map $\Gamma$ has a non-unique definition. To show geometric ergodicity, we rely on differentiability properties of this map $\Gamma$. Hence we must rely on a rigid definition of the map $\Gamma$ involving a fixed choice of singular value decomposition. From here on, any function $\Gamma$ that fits into the above formulation will be called an \emph{EAKF update map}. The following theorem, which is the main result of the section states that there exists a choice of EAKF update map that renders the EAKF process geometrically ergodic. Due to the complexity in defining $\Gamma$,  statements concerning the ergodicity of an arbitrary EAKF algorithm seem out of reach.


\begin{thm}
\label{thm:EAKFergo}
There exists an EAKF update mapping $\Gamma$ such that if the signal-ensemble process generated by it has a strong Lyapunov function, and the unfiltered signal process $U_n$ has an kinetic energy principle with nondegenerate system noise, in other words Assumptions \ref{aspt:kinetic}, \ref{aspt:Lyapunov} and \ref{aspt:density} hold,  then the EAKF signal-ensemble process is geometrically ergodic in total variation distance.
\end{thm}


\subsubsection{The choice of intermediate point}
Unlike the ETKF case, we cannot pick the intermediate point to be the origin. This is because at the origin, the spectrum of $\Shat^T H^TH\Shat$ clusters at $0$ and a perturbation of $\Shat$ may split the eigenvalues into branches while leaving the eigenvector basis matrices $G$ and $R$ to be singular \cite{Kat82}. We will instead choose an intermediate point that has simple nonzero spectrum.

The intermediate point $y^* = (y_1^*,y_2^*)$ will be defined using a matrix $\Mmax$ constructed below. Precisely, we choose  
\begin{equ}\label{e:EAKF_int}
y_1^* = (\Ustar,\Vhat^{*(1)},\dots,\Vhat^{*(K)}) =  (0,\Mmax^{(1)},\dots,\Mmax^{(K)}) \quad y_2^* = \Zstar = 0 \;,  
\end{equ}
where $\Mmax^{(j)}$ denotes the $j$-th column of $\Mmax$. By construction of $\Mmax$, this choice will satisfy $\overline{\widehat{V}^*} = 0$ and hence $\Shat^* = \Mmax$ and moreover will attain the highest possible rank for such a matrix. As we shall see, these properties ensure that the EAKF map is (locally) well behaved under perturbations. 
\par
In the sequel, we assume that the observation matrix $H$ is diagonal with rank $q \leq d$. Using the change of coordinates described in Remark \ref{rmk:svd} as well as the fact that Assumptions \ref{aspt:Lyapunov}, \ref{aspt:density} and the statement of geometric ergodicity hold equivalently in both coordinate systems, this can be achieved without loss of generality.



\begin{lem}
\label{lem:onepoint}
If $M$ is a $d\times K$ matrix such that $M\vec{1}=0$ then the rank of $M$ is at most $r:=\min\{K-1,d\}$. Moreover, if we define the $d \times K$ matrix $\Mmax$ by 
\begin{equ}
\Mmax := \begin{cases} \begin{bmatrix} \vec{1} & \Xi_0 & 0 \end{bmatrix}  \quad \text{if $d \leq K-1$} \\  \\ 
\begin{bmatrix} \vec{1} &  \Xi_0 \\ 0 &  0 \end{bmatrix}  \quad \text{if $d \geq K-1$}  
\end{cases}
\end{equ}
where $\vec{1}$ is the $d\times 1$ vector of $1$'s and $\Xi_0$ is the $r \times r$ matrix 
\begin{equ}
\Xi_0 = \begin{bmatrix} 1 & 1 & \dots & 1 & -r \\
\vdots & \vdots & \vdots & \iddots & 0 \\
1 & 1 & -3 & \dots & 0\\ 
1 & -2 & 0 & \dots & 0 \\
-1 & 0 &  0 & \dots & 0 
\end{bmatrix}
\end{equ} 
%
then $M_0$ has the following properties
\begin{enumerate} 
\item[(i)]  $\Mmax \vec{1} = 0$, ${\rm{rank}}(\Mmax) = r$ and all nonzero eigenvalues of $\Mmax$ are simple.  
\item[(ii)] $\Mmax$ has the SVD $\Mmax = \Qmax \Lambdamax \Rmax$ where $\Qmax$ is the $d\times d$ identity matrix and the last row of $R_0$ is $\vec{1}^T$.
\item[(iii)] When $H$ is a diagonal matrix with descending diagonal entries, $(K-1)^{-1}\Lambdamax\Qmax^T H^T H \Qmax\Lambdamax$ has eigen-decomposition $G_0^T \Sigmamax G_0$, where $\Sigmamax$ has descending diagonal entries and is of rank $r_H : = \min (q,K-1)$, and $G_0$ is the $r\times r$ identity matrix.
\end{enumerate} 
%
%
\end{lem}
\begin{proof}
It is elementary to see that rank$(M)\leq d$. Also, since $M$ has the same rank as  $M^TM$, while $M^TM$ has $\vec{1}$ as a null right vector, so rank$(M)\leq K-1$. For the claims for $\Mmax$, direct verification will be sufficient, where  the orthogonality between different rows of $\Mmax$ easily leads to
\[
\Mmax\Mmax^{T}=\text{diag}\{ r(r+1),\ldots, 2\times 3, 1\times 2,0\ldots\}
\]
Clearly their spectrums are as requested. The claims for $G_0$ and $\Sigmamax$ easily follow from direct verification. 
\end{proof}

\subsubsection{Differentiability of EAKF near the intermediate point}
The intermediate point $y^*$ was chosen to ensure stability of the eigenvalues of $\Shat^T \Shat$ and $\Lambda_1 Q_1^T H^T H \Lambda_1 Q_1$ as a function of $y$, in a neighborhood of the intermediate point. We will now demonstrate this.  

%

\begin{lem}
\label{lem:diffEAKF}
We can construct an EAKF update map $\Gamma$ which is continuously differentiable in a neighborhood of the intermediate point $y^*$. 
\end{lem}
\begin{proof}
Recall that $\Gamma$ is defined by $\Gamma = (U,V^{(1)},\dots,V^{(K)})$ where
\[
V^{(k)} = \Vbar + S^{(k)}
\]
for each $k =1,\dots,K$ where $S^{(k)}$ is the $k$-th column of $S$ and 
\[
S=Q_1\Lambda_1 G_1^T(I+D_1)^{-1/2}R_1, \quad \overline{V}=\overline{\widehat{V}}-\widehat{C}H^T(I+H^T\widehat{C}H)^{-1}[H\overline{\widehat{V}}-Z]\;.
\]
The only unspecified parts of the map are the matrices $Q,\Lambda, R$ and $G$ (and thus $Q_1,\Lambda_1, R_1$ and $G_1$), hence these must be constructed. It suffices to show that all terms appearing in the above map, including the constructed terms $G_1,R_1$, are continuously differentiable as a function of $y = (U,\Vhat^{(1)},\dots,\Vhat^{(K)},Z)$ in some neighborhood of $y^*$. 
\par
Clearly $\overline{V}$, $\Shat$ and $\Shat^T\Shat$ depend smoothly on $y$, differentiability of $\overline{V}$ is shown in the proof of Theorem \ref{thm:EnKFergo}. The latter two can be directly seen from the definitions of the mean and spread update maps in \eqref{sys:ESRFmean} and \eqref{sys:EAKFspread}. 

We claim that there is a ${C}^1$ extension of $Q, R$ as $ Q(y), R(y)$ in a neighborhood of $y=y^*$ such that
\begin{enumerate}
\item $Q(y)$ and $R(y)$ are the eigenbasis matrices in the decomposition of $\Shat(y) \Shat^T(y)$ and $\Shat^T(y) \Shat(y)$ respectively.
\item There is a diagonal matrix $\Lambda(y)$ such that $\Shat(y) = Q(y) \Lambda (y) R^T(y)$. 
\end{enumerate} 
To verify the first claim, note that all $r$ nonzero eigenvalues of $\widehat{S}^T(y)\widehat{S}(y)$ are by construction simple at $y^*$, hence they are smooth with respect to $y$ by II.2.2 of \cite{Kat82}, and in particular stay positive, distinct and maintain the same descending order for $y$ near $y^*$. Moreover, because the rank of  $\widehat{S}^T(y)\widehat{S}(y)$  is at most $r$, $0$ will be an eigenvalue of multiplicity $K-r$ for $y$ near $y^*$. By Lemma \ref{lem:eigentransform}, there is a transformation matrix $U(y)$ that transforms the simple nonzero eigenvectors and null space of $\Shat^T(y^*)\Shat(y^*)$ to the ones of $\Shat^T(y)\Shat(y)$. So $R(y):= R_0U(y)^T$ provides an eigenbasis matrix for $\Shat(y)^T\Shat(y)$ for $y$ near $y^*$. Likewise we can also find a $Q(y)$ as the eigenbasis matrix for $\Shat(y) \Shat^T(y)$ for $y$ near $y^*$. 
\par
We will now check the second point of the claim. If there is another SVD, $\widehat{S}(y)=\widetilde{Q}\widetilde{\Lambda}\widetilde{R}$ at $y$ close to $y^*$ with $\widetilde{\Lambda}$ having descending eigenvalues, then $\widetilde{R}$ has its first $r$ rows being the eigenvectors of $\widehat{S}(y)^T\widehat{S}(y)$ associated with descending nonzero eigenvalues, and the remaining $K-r$ rows correspond to the basis of the null space. Therefore, $R(y)$ has the first $r$ rows being either identical to the ones of $\widetilde{R}$ or their additive inverse. Likewise we have the same conclusion for $Q(y)$. Then because the choice of eigenvectors for the null space are unimportant, which can be told from the fact that  $Q_2,R_2$ play no role in \eqref{eqn:SVD}, we have
\[
\Lambda(y):=Q(y)^T\widehat{S}(y)R(y)^T
\]
is diagonal, and each component is either the same as $\widetilde{\Lambda}$ or the additive inverse. Yet $\Lambda(y)$ is continuous with respect to $y$, so for there exists a neighborhood of $y^*$ so that they remain nonnegative, this indicates $\Lambda(y)=\widetilde{\Lambda}$ and completes the proof of the two claims. 

Similarly, we can find a $C^1$ function $G_1(y)$ defined in a neighborhood of $y^*$ such that $G_1(y^*) = I_r$ and such that $G_1(y)$ is an eigen-basis matrix for $\Lambda_1(y) Q_1(y)^T H^T H Q_1(y)\Lambda_1(y)$ with corresponding diagonal matrix $D_1(y)$, as in part $(ii)$ of the EAKF formulation.

Finally, the function $(I+D_1)^{1/2}$ is always $C^1$ in $y$. To see this, note that 
\[
D_1=(K-1)^{-1}G_1\Lambda_1 Q_1^TH^THQ_1\Lambda_1 G_1^T,
\] 
which is ${C}^1$ in $y$. But since $D_1$ is diagonal with entries nonnegative, it follows that $(I+D_1)^{-1/2}$ must also be $C^1$ in $y$. Hence all terms involved in $\Gamma$ are $C^1$ in a neighborhood of $y^*$ and the proof is complete.

\end{proof}

\subsubsection{Controllability of EAKF at the intermediate point}
\begin{lem}
\label{lem:controlEAKF}
The EAKF update map $\Gamma$ constructed in Lemma \ref{lem:diffEAKF} has its Jacobian at $y^*$ being non-degenerate. 
\end{lem}
\begin{proof}
Before proceeding, we recall some notation. Let $\Shat=\Shat(y)$ be the spread matrix constructed from $y$, similarly let $Q_1,\Lambda_1,R_1,G_1,D,D_1$ be the matrix valued functions of $y$ that were constructed for the purposes of the EAKF map in Lemma \ref{lem:diffEAKF}. In the proof below, we will frequently use the subscript $0$ to denote a matrix valued function being evaluated at the intermediate point $y=y^*$. 
\par
Recall from Lemma \ref{lem:onepoint} that the SVD of $\Shat(y^*) = M_0$ is given by 
\[
\widehat{S}(y^*)=I_d \Lambda_0 R_0=I_d \begin{bmatrix}\Lambda_{0,1} &0\\ 0 &0\end{bmatrix}\begin{bmatrix}R_{0,1}\\R_{0,2}\end{bmatrix}.
\]
Since the last row of $R_{0,2}$ is $\vec{1}$ by Lemma \ref{lem:onepoint}, we can also write the SVD decomposition as 
\[
I_d \Lambda_0 R_0=I_d \begin{bmatrix}\Lambda_{S} &0\\ 0 &0\end{bmatrix}\begin{bmatrix}R_{S}
\\ \vec{1}^T\end{bmatrix},
\]
where $R_S$ is the first $K-1$ rows of $R_0$, and $\Lambda_S$ is the upper $(K-1)\times (K-1)$ sub-block of $\Lambda_0$.  We are only interested in the derivatives of $\Gamma$ in the directions $y_1 = [U, \Vhat^{(1)},\ldots, \Vhat^{(K)}]$, these directions clearly form a $d(K+1)$ dimensional subspace. We can always write this subspace as 
%
%
%
\[
\mathcal{M} : = \{[u, \overline{v}\otimes \vec{1}_K+BR_S]: u,\overline{v}\in \mathbb{R}^d, B\in \mathbb{R}^{d\times (K-1)}\}.
\] 
Indeed, it suffices to see that any spread matrix $\Shat$ can be decomposed $\Shat = B R_S$ for some suitable $B$. 


To prove the lemma, suppose that there exists a direction $M \in \CM$ such that the derivative of $[U, V^{(1)},\dots,V^{(K)}]$ in the direction $M$ vanishes at $y^*$. We will show that $M$ must be zero. Write $M = [u,\overline{v}\otimes \vec{1}_K+BR_S]$. If the derivative vanishes at $y^*$, we must also have   
\[
\mathcal{D}_MU=\mathcal{D}_M \overline{V}=0,\quad \mathcal{D}_MS=0. 
\]
at $y^*$. Through the following steps, we will show that $M$ must be zero. In steps $1-3$ we show $u = \vbar = 0$ and in steps $4)-9)$ we show that $B = 0$.  

\begin{enumerate}[Step 1)]
\item $u=0$: This is a direct consequence of $\mathcal{D}_M U=u$.
\item If $\widehat{C}=\Shat\Shat^T$ then $\mathcal{D}_M \widehat{C}=0$ at $y^*$: Note that since $\mathcal{D}_M S=0$ at $y^*$, 
\[
\mathcal{D}_M [\widehat{C}-\widehat{C}H^T(H^T\widehat{C}H+I)^{-1}H\widehat{C}]
=\mathcal{D}_M(S S^T)=0,
\]
where $\widehat{C}=\widehat{S}\widehat{S}^T$ is clearly differentiable. Using the formula of the derivative of an inverse, we find the previous equation is equivalent to 
\begin{align}
\notag
0=&\mathcal{D}_M\widehat{C}-\mathcal{D}_M\widehat{C}H^T(H^T\widehat{C}H+I)^{-1}H\widehat{C}
-\widehat{C}H^T(H^T\widehat{C}H+I)^{-1}H\mathcal{D}_M\widehat{C}\\
\label{tmp:DMC}
&+\widehat{C}H^T(H^T\widehat{C}H+I)^{-1}H^T\mathcal{D}_M\widehat{C}H(H^T\widehat{C}H+I)^{-1}H\widehat{C}.
\end{align}
Now let $c_{ij}$ denote the entries of $\mathcal{D}_M \widehat{C}$ and let $\lambda_i$, $h_i$ denote the diagonal entries of the diagonal matrices $\Lambda_0$ and $H$ respectively. Since the right hand side of \eqref{tmp:DMC} has all matrices being diagonal except $\mathcal{D}_M \widehat{C}$, we can compute it explicitly, and find the $ij$-th entry is 
\[
(\lambda^2_ih_i^2+1)^{-1}(\lambda^2_jh_j^2+1)^{-1}c_{ij}.
\] 
Therefore \eqref{tmp:DMC} implies that $\mathcal{D}_M \widehat{C}=0$ at $y^*$.
 \item  $\overline{v}=\vec{0}$: Since $\mathcal{D}_M \overline{V}=0$ and $\overline{V}$ is given by 
\[
\overline{V}=\overline{\widehat{V}}-\widehat{C}H^T(I+H^T\widehat{C}H)^{-1}[H\overline{\widehat{V}}-Z]
\]
so using the fact that $\mathcal{D}_M\widehat{C}=0$ at $y^*$,
\[
\mathcal{D}_M \overline{V}=\mathcal{D}_M \overline{\widehat{V}}-\widehat{C}H^T(I+H^T\widehat{C}H)^{-1}[H\mathcal{D}_M\overline{\widehat{V}}]
\]
which at $y^*$ can be computed explicitly because the matrices are diagonal:
\[
(\mathcal{D}_M\overline{V})_i=(1+h_i^2\lambda_i^2)^{-1/2}(\mathcal{D}_M \overline{\widehat{V}})_i\quad \Rightarrow \quad 0=\mathcal{D}_M\overline{\widehat{V}}_i=\overline{v}. 
\]

\item $\lambda_jB_{ij}+\lambda_iB_{ji}=0$ and the diagonal terms of $B$ are zero: From here and after, the range of index $(i,j)$ is in $\{1,\ldots,d\}\times\{1,\ldots,K-1\}$. Our claim comes from the fact that $\widehat{C}=\Shat\Shat^T$ and $\mathcal{D}_M \Shat=BR_S$ so at $y = y^*$ we have
\[
0=\mathcal{D}_M\widehat{C}=(\mathcal{D}_M\widehat{S}) \widehat{S}^{T}(y^*)+\widehat{S}(y^*)(\mathcal{D}_M\widehat{S})^T=BR_0 \widehat{S}^{T}(y^*)+\widehat{S}(y^*) R_0^T B^T=B\Lambda^T_0 +\Lambda_0 B^T. 
\]
Then notice that $B$ is $d\times (K-1)$ dimensional, so its diagonal terms correspond to the first $r=\min\{d,K-1\}$ nonzero entries of $\lambda_i$, therefore they are zero.
\item $\mathcal{D}_M Q_1=0, \mathcal{D}_M \Lambda_1=0$ at $y^*$: Because $Q_1$ is formed by eigenvectors  of matrix $\widehat{C}$ associated with the nonzero simple eigenvalues, which are the entries of $\Lambda_1$, by Lemma \ref{lem:eigentransform} and formula \eqref{eqn:perturbU} we have our claim because $\mathcal{D}_M \widehat{C}=0.$
\item $[\mathcal{D}_M(R_1)R_{S}^T]_{ij}=(\lambda_iB_{ij}+\lambda_jB_{ji})/(\lambda_i^2-\lambda_j^2)$ at $y^*$ for $i\neq j$: by Lemma \ref{lem:eigentransform}, if we denote $r_i$ to be the  $i$-th row of $R_{0,1}$, while $\Psi_i$ being $e_i^T r_i$, then because $R_{0,2}^TR_{0,2}$ is the projection to the null space of $\Shat^T\Shat$ at $y^*$,
\[
\mathcal{D}_M r_i=r_i\mathcal{D}_M (\widehat{S}^T\widehat{S})\bigg[\lambda_i^{-2}R_{0,2}^TR_{0,2}+\sum_{k\neq i}^r(\lambda_i^2-\lambda_k^2)^{-1}\Psi_k^T\Psi_k\bigg].
\]
So for $1\leq j\leq r$, because $\Psi_k^T\Psi_k r_j^T=1_{k=j}r_j^T$, $R_{0,2}^TR_{0,2}r_j^T=0$, we find that 
\[
[\mathcal{D}_M(R_1)R_{S}^T]_{ij}
=\mathcal{D}_M(r_i)r_j^T=(\lambda_i^2-\lambda_j^2)^{-1}r_i\mathcal{D}_M (\widehat{S}^T\widehat{S})r_j^T. 
\]
For $r\leq j\leq K-1$, because $\Psi_k^T\Psi_k r_j^T=0$, $R_{0,2}^TR_{0,2}r_j^T=r_j^T$, we find that 
\[
[\mathcal{D}_M(R_1)R_{S}^T]_{ij}
=\mathcal{D}_M(r_i)r_j^T=\lambda_i^{-2}r_i\mathcal{D}_M (\widehat{S}^T\widehat{S})r_j^T. 
\]
Then using $r_iR_{S}^T=e_i$, the row with zero component except being $1$ on $i$-th coordinate, 
\[
r_i\mathcal{D}_M(\Shat^T\Shat)r_j^T=r_i(R_{S}^TB^T\Lambda_0 R_{S}+R_{S}^T\Lambda_0^TB R_{S})r_j^T
=\lambda_j e_iB^Te_j^T+\lambda_i e_iBe_j^T,
\]
which concludes our claim.

\item $\mathcal{D}_M G_1=0, \mathcal{D}_M D_1=0$ at $y^*$: Because $G_1^TD_1G_1$ is the eigenvalue decomposition of $(K-1)^{-1}\Lambda_1G_1^TH^THG_1\Lambda_1$, due to Step 5), the derivative in the direction $M$ must vanish at $y^*$, so by Lemma \ref{lem:eigentransform} we have our claim.

\item $\mathcal{D}_M D=0$ and $\mathcal{D}_M D_1=0$ at $y^*$: recall that $G^TDG$ is the eigen-decomposition of $\Shat H^TH\Shat$ and $D_1$ is the upper $r\times r$ block of $D$. Notice that $D$ has  its entries being the eigenvalues of $\Shat^T H^TH\Shat$, which is the same as $H^T\Shat\Shat^TH=H^T\Chat H$, but from Step 2) $\mathcal{D}_M\Chat=0$ at $y^*$, so we have our claim. 

\item $B=0$: By step 4), we only need to verify non-diagonal entries. From $\mathcal{D}_M S=0$, $S=Q_1\Lambda_1G_1^T(I+D_1)^{-1/2}R_1$ and results from Steps 4), 6), 7), we have
\[
0=Q_1\Lambda_1 G_1^T(I+D_1)^{-1/2}\mathcal{D}_M R_1
\]
which by $Q_1\Lambda_1=\Lambda_1$ which is invertible at $y^*$ and $G_1=I_r$ at $y^*$ leads to 
\[
0=(I+D_1)^{-1/2}(\mathcal{D}_M R_1)R_S^T
\]
at $y^*$. Because  $(I+D_1)^{-1/2}$ is a diagonal matrix with entries $(1+\lambda_i^2h_i^2)^{-1/2}$, using the results from Step 6),  the $(i,j)$-th entry of the right hand side is
\[
\frac{\lambda_iB_{ij}+\lambda_jB_{ji}}{\sqrt{1+\lambda_i^2h_i^2}(\lambda_i^2-\lambda_j^2)},
\]
 so $\lambda_iB_{ij}+\lambda_jB_{ji}=0$ for $i\neq j$. We claim that $\lambda_i\neq \lambda_j$ for $i\neq j$, this is because $\lambda_i=\lambda_j$  only when they are both zeros, yet either $i\leq d=r$ or $j\leq K-1=r$, and hence $\lambda_i$ or $\lambda_j$ is not zero.  Recall that from Step 4) we have that $\lambda_jB_{ij}+\lambda_iB_{ji}=0$, combining it with our results that $\lambda_iB_{ij}+\lambda_jB_{ji}=0$ and $\lambda_i\neq \lambda_j$, this can only be possible when $B_{ij} = 0$. 
\end{enumerate}
Therefore, the Jacobian of mapping $\Gamma$ can not be degenerate at $y^*$. 
\end{proof}
We now have the ingredients required to prove our Theorem \ref{thm:EAKFergo}

\begin{proof}[Proof of Theorem \ref{thm:EAKFergo}]
In light of Lemmas \ref{lem:diffEAKF} and \ref{lem:controlEAKF}, the result follows identically to the proof of Theorem \ref{thm:ETKFergo}. 
\end{proof}

\section{Conclusion and Discusion}
\label{sec:conclude}
In the preceding pages, we have established a simple analytic framework for validating two important nonlinear stability properties of ensemble based Kalman filters, namely boundedness and geometric ergodicity for the signal-ensemble process. These are important properties in practice, as they guarantee the filter processes will remain bounded on an infinite time horizon and that initialization errors in the algorithm dissipate exponentially quickly in time. 
\par
In Section \ref{sec:stabobs}, upper bounds for the signal-ensemble processes are obtained via a simple Lyapnuov function argument. In particular we show that, if the signal satisfies the \emph{observable energy criterion}, introduced in Assumption \ref{ass:energy_obs}, then one can construct a Lyapnuov function for the signal-ensemble process. The sub-level sets for this Lyapnuov function are only compact in the observed directions. Hence this can be thought of as an upper bound for the observable ensemble $\{H V^{(k)}_n\}_{k=1}^K$. This Lyapnuov function argument is used to construct upper bounds for EnKF, ETKF and  EAKF.
\par
Section \ref{sec:range} discusses the applicability  of the observable energy criterion. Heuristically speaking, systems with strong dissipation in kinetic energy and complete observations as well as suitable spatial observations will satisfy the observable energy criterion,  therefore their EnKF and ESQF ensemble will be bounded uniformly in time. On the other hand, systems without observable energy criterion have the potential to diverge to machine infinity in finite time, which is known as  catastrophic filter divergence. This phenomenon has been captured in previous numerical experiments \cite{MH08, MH12, GM13}. In a separate article, the authors have constructed a concrete nonlinear system with kinetic energy dissipation but without observable energy dissipation, whose EnKF ensemble diverges to infinity exponentially fast with large probability \cite{KMT15}.
\par
In Section \ref{sec:ergo}, geometric ergodicity is established for the signal-ensemble processes of EnKF, ETKF and EAKF. This is achieved by combining the existence of a Lyapnunov function (with compact sub-level sets) with a minorization condition. The existence of a Lyapnuov function can be guaranteed using the results from Section \ref{sec:ergo}, but any other Lyapnuov function would suffice.    
\par

%

While our results have shown that systems with dissipative observable energy will produce stable filter processes, it is hard to believe that Assumption \ref{aspt:energydisc} is necessary for stability. 
As discussed earlier, numerical evidence suggests that ensemble based Kalman filters are typically  very stable and catastrophic filter divergence only happens when the observations are sparse. The stability and ergodicity of the filter processes here are inherited from the original nonlinear system. This type of inheritance phenomenon generally exists for many filter processes, for example Kalman filter preserves linear stability, and optimal filters preserves absolute regularity of general Markov processes \cite{Jaz72,RGY99, TvH12, TvH14}. In order to extend our results, one might seek new dynamical properties that can be inherited through the Kalman assimilation step. 
\par
In obtaining our results, we use few properties of the forecast covariance matrix other than positivity. On the one hand this lends generality to our results, in that the upper bounds hold for a broad class of ensemble methods, including covariance inflation schemes. On the other hand, if we had more control over the covariance one might hope to gain more control over the filter ensemble. Thus, it is quite natural to ask whether  one can ``inflate'' the covariance adaptively in order to guarantee stability properties, even in the case where the observation $H$ is of low rank. This question will be investigated by the authors in a subsequent paper \cite{TMK15}.  Similar nonlinear stability and ergodicity results as developed here are also valid for finite ensemble Kalman filters which are continuous in time \cite{KLS14} and will be reported elsewhere. 
\par
%
%
%
\ack
This research is supported by the MURI award grant N-000-1412-10912, where A.J.M. is the principal investigator, while X.T.T. is supported as a postdoctoral fellow. D.K. is supported as a Courant Instructor. We thank Ian Grooms and Tom Trogdon for their discussions over some topics of this paper.

\appendix
\section{Elementary claims}
\begin{lem}
\label{lem:inver}
Let $A$ be a positive semidefinite symmetric matrix, then the following holds:
\[
0\preceq A(A+I)^{-1}\preceq I,\quad 0\preceq (A+I)^{-1}\preceq I,\quad 0\preceq (A+I)^{-1}\preceq A^{-1}\;,
\]
where $A \preceq B$ means that $B-A $ is positive semidefinite. 
\end{lem}
\begin{proof}
Since $A(A+I)^{-1}+(A+I)^{-1}=I$, it suffices to show $0\preceq (A+I)^{-1}\preceq I$. Since $A$ is positive semidefinite and symmetric, it can be diagonalized through an orthogonal matrix $\Psi,$ i.e. $A=\Psi D \Psi^T$ with $D$ being diagonal. Then based on $(A+I)^{-1}=\Psi (D+I)^{-1}\Psi^T$, it is elementary to conclude our lemma. 
\end{proof}
\begin{lem}
\label{lem:young}
By Young's inequality, for any $\epsilon>0$ and $x,y\in \mathbb{R}^d$, the following holds:
\[
|x+y|^2=|x|^2+|y|^2+2\langle x, y\rangle\leq (1+\epsilon^2) |x|^2+(1+\epsilon^{-2})|y|^2. 
\]
\end{lem}

\begin{lem}
\label{lem:trinv}
For any two $d\times d$ positive semidefinite symmetric matrices $A$ and $B$, the following holds:
\[
\text{tr}(A(A+B)^{-1})=\text{tr}((A+B)^{-1}A)
\leq \text{tr}((A+B)^{-1}(A+B))= d. 
\]
\end{lem}

\begin{lem}
\[
S=\widehat{S}[I_K+(K-1)^{-1}\widehat{S}^TH^TH\widehat{S}]^{-1/2}\Theta(\vec{F})[\widehat{S}\widehat{S}^T]^{-1/2}\widehat{S}
\]
\end{lem}
\begin{proof}
Let $N$ be a matrix with its columns form an orthogonal basis of the null space of $\widehat{S}$, i.e. 
\[
N=[e_1,\ldots, e_u], \quad \widehat{S}N=0,\quad n+\text{rank}(\vec{F})=d. 
\] 
Then we can find $F$ and $X$ such that
\[
\widehat{S}\widehat{S}^T=[N,F]\begin{bmatrix}0&0\\ 0 &\Sigma \end{bmatrix}\begin{bmatrix}N^T\\F^T\end{bmatrix},\quad
I_K+(K-1)^{-1}\widehat{S}^TH^TH\widehat{S}=[N,F]\begin{bmatrix}0&0\\ 0 &\Sigma \end{bmatrix}\begin{bmatrix}N^T\\F^T\end{bmatrix}
\]
\end{proof}
\section{Discrete time formulation}
\label{sec:discrete}
In many applications, the underlying dynamical system is given by an ODE or SDE, which in general can be written as 
\[
du_t=\psi(u_t)+\Sigma dW_t.
\]
However, the noisy observations of these systems are usually made not in continuous time, but rather sequentially in time, with time interval $h$. Therefore it is natural to see the stochastic process instead as a stochastic sequence $U_n=u_{nh}$ as we do in Section \ref{sec:model}. Theoretically speaking, we can always transform the SDE for $u_t$ into the nonlinear update map of $U_n$ as in \eqref{sys:discrete}, since if we write the transition kernel of process $u_t$ from time $0$ to time $h$ as $K_h(u, dv)$, then it suffices to let 
\[
\Psi_h(u)=\mathbb{E}(u_h|u_0=u)=\int vK_h(u,dv),\quad \zeta_n=u_{nh}-\Psi_h(u_{(n-1)h}).  
\]
The reader should notice that the introduction of the nonlinear map $\Psi_h$ and random sequence $\zeta_n$ is just for the convenience of  our formal illustration and rigorous proof. In order to do the forecast step of EnKF or ESQF in practice, it is not necessary to find the concrete formulation of  $\Psi_h$; it suffices to simulate the SDE starting from each posterior ensemble $V_{n-1}^{(k)}$ from time $0$ to $h$, and let the forecast ensemble $\Vhat^{(k)}_n$ be the  realization of the simulation at time $h$. 

It is also easy to obtain an energy principle, Assumption \ref{aspt:kinetic}, for the discrete time version $U_n=u_{nh}$, as long as the original SDE satisfies that $\langle u, \psi(u)\rangle\leq -\gamma |u|^2+k$ for some $\gamma, k>0$. Because then the generator of the square norm $|u_t|^2$ would satisfy
\[
\mathcal{L}|u_t|^2=2\langle u_t, \psi(u_t)\rangle+\text{tr}(\Sigma\Sigma^T)\leq -2\gamma|u_t|^2+2k+\text{tr}(\Sigma\Sigma^T),
\]
which by Gr\"{o}nwall's inequality and Dynkin's formula yields
\begin{align*}
\mathbb{E} |u_h|^2&\leq e^{-2\gamma h}|u_0|^2+\int^h_0 e^{-2\gamma (h-s)}(2k+\text{tr}(\Sigma\Sigma^T)) ds\\
&\leq e^{-2\gamma h}|u_0|^2+h(2k+\text{tr}(\Sigma\Sigma^T)).
\end{align*}

Moreover, when the stochastic covariance matrix $\Sigma$ is nonsingular, by \cite{SV72} the transition kernel $K_h(u,dv)$ is equivalent to the Lebesgue measure on $\mathbb{R}^d$; also by Theorem 38.16 of \cite{RW00}, this density is smooth both in $u$ and $v$. As a consequence, the nondegnerate system noise condition, Assumption \ref{aspt:density}, holds because $\{(u,v):|u|\leq M_1, |v|\leq M_2\}$  is compact. More generally, Assumption \ref{aspt:density} should hold for discrete time formulation of hypoelliptic systems where certain controllability conditions are satisfied, please refer to \cite{MS02} and the reference therein.

\section{EAKF with rank deficiency}
\label{sec:EAKFpatho}
In this section, we demonstrate that, when there is rank deficiency, in order to achieve the posterior covariance relation \eqref{eqn:postprior} for EAKF, it is necessary to use specific eigen-decompositions in the formulation of EAKF. In particular, certain choices of eigen-decomposition will result in a violation of \eqref{eqn:postprior}. We will also show that the formulation employed in Section \ref{sec:formulationEAKF} does always satisfy \eqref{eqn:postprior}. 
\par
The following example illustrates that not all eigen-decompoisiton will satisfy \eqref{eqn:postprior}. Let
\[
\Shat=\begin{bmatrix}
1 &-1\\
0 & 0
\end{bmatrix},\quad 
H=\begin{bmatrix}
0 &0\\
0 &0
\end{bmatrix}\;,
\]
then the SVD of $\Shat$ is given by
\[
\Shat=Q\Lambda R=\begin{bmatrix}
1 &0\\
0 & 1
\end{bmatrix}
\begin{bmatrix}
\sqrt{2} &0\\
0 & 0
\end{bmatrix}
\begin{bmatrix}
\frac{1}{\sqrt{2}} &\frac{-1}{\sqrt{2}}\\
\frac{1}{\sqrt{2}} &\frac{1}{\sqrt{2}}
\end{bmatrix}
\]
Notice that $\Lambda Q^T H^THQ \Lambda=0$, so $G$ can be chosen as any orthonormal matrix. One possible choice of matrix $G$ is $R^T$ and $D=0$. However, the spread update with $G=R^T$ following the classical formulation is
\[
S=Q\Lambda G^T(I+D)^{-1/2}\Lambda^\dagger Q^T\Shat=\begin{bmatrix}
1 &-1\\
0 & 0
\end{bmatrix} \begin{bmatrix}
\frac{1}{\sqrt{2}} & 0\\
0 & 0
\end{bmatrix}I_2\begin{bmatrix}
1 &-1\\
0 & 0
\end{bmatrix}
=\begin{bmatrix}
\frac{1}{\sqrt{2}} &\frac{-1}{\sqrt{2}}\\
0 & 0
\end{bmatrix}.
\]
This posterior spread $S$ does not match the posterior covariance relation \eqref{eqn:postprior}, which requires that $SS^T=\Shat \Shat^T$ since $H=0$ indicates there is no observation.  It is necessary for us to choose $G=I_2$ here, which produces the right posterior covariance, 
\[
S=Q\Lambda G^T(I+D)^{-1/2}\Lambda^\dagger Q^T\Shat=\begin{bmatrix}
\sqrt{2} & 0\\
0 & 0
\end{bmatrix} I_2 \begin{bmatrix}
\frac{1}{\sqrt{2}} & 0\\
0 & 0
\end{bmatrix}I_2\begin{bmatrix}
1 &-1\\
0 & 0
\end{bmatrix}
=\begin{bmatrix}
1 &-1\\
0 & 0
\end{bmatrix}. 
\]
In general, as long as $G_2R_1^T=R_2 G_1^T=0$ holds in the formulation of EAKF, then the posterior covariance relation \eqref{eqn:postprior} holds with formulation \eqref{sys:EAKFspreaddeficient}. One can see that from 
\[
SS^T=Q_1\Lambda_1 G^T_1(I+D_1)^{-1}G_1\Lambda_1Q_1^T=
Q_1\Lambda_1[I+\Lambda_1 Q_1^TH^THQ_1\Lambda_1]^{-1}\Lambda_1Q_1^T
\]
which satisfies our need because of the following algebraic identity with $A=Q_1\Lambda_1$
\[
[I+(K-1)^{-1}A^TH^THA]^{-1}=I-(K-1)^{-1}A^TH^T(I+HAA^T H^T)^{-1}HA. 
\]

On the practical side, in mathematical computing packages, directly applying  the classical EAKF formulation \eqref{sys:EAKFspread} will produce the same result as the formula we used \eqref{sys:EAKFspreaddeficient}. To see this, we note that 
\[
\Lambda^T Q^TH^THQ\Lambda =\begin{bmatrix}
\Lambda_1 &0\\
0 &0
\end{bmatrix}
\begin{bmatrix}
Q_1^T\\
Q_2^T
\end{bmatrix}H^TH\begin{bmatrix}Q_1 &Q_2\end{bmatrix}
\begin{bmatrix}
\Lambda_1 &0\\
0 &0
\end{bmatrix}=\begin{bmatrix}
\Lambda_1Q_1^TH^THQ_1\Lambda_1 &0\\
0 &0
\end{bmatrix}.
\]
In the default setting of most mathematical programs, the eigenvalue decomposition will arrange the eigenvalues in decreasing order, and due to the block structure, producing the orthonormal matrix $G$ and eigenvalue matrix $D$ to be
\[
G=\begin{bmatrix}
G_1 &0\\
0 & I_{K-r}
\end{bmatrix},\quad 
D=\begin{bmatrix}
D_1 &0\\
0 & 0
\end{bmatrix}. 
\]
Then $S$ is given by the following,
\[
Q\Lambda G^T(I+D)^{-1/2}\Lambda^\dagger Q^T\Shat
=[Q_1\Lambda_1 \quad 0]\begin{bmatrix}
G_1^T &0\\
0 & I
\end{bmatrix}
\begin{bmatrix}
(I+D_1)^{-1/2} &0\\
0 & I
\end{bmatrix}
\begin{bmatrix}
\Lambda_1^{-1} &0\\
0 & I
\end{bmatrix}
\begin{bmatrix}
Q_1^T\\
Q_2^T
\end{bmatrix}Q_1\Lambda_1R_1,
\]
which is the same as $Q_1\Lambda_1G_1^T(I+D_1)^{-1/2}R_1$.

 \section{Perturbation theory for matrices}
\label{sec:perturb}
In \cite{Kat82}, the perturbation theory for matrices are carefully studied. Here we collect some useful results for our study of the controllability of ESRF, where all the section numbers and page numbers mentioned below are referring to \cite{Kat82} if not otherwise specified. One thing the reader must be careful is that most results in \cite{Kat82} are for univariate perturbation, while our interest lies in multivariate perturbation. So we generally follow the strategy of Section II.5.7, that is trying to express matrices through contour integral of the resolvent, and restrict to cases with stable  spectrum structure when eigenprojections are involved (hence why we pick a specific $\vec{F}^*$ in Lemma \ref{lem:onepoint}).

The most fundamental tool used in \cite{Kat82} is the resolvent of a symmetric positive semi-definite matrix $C$,  given in Section I.5.2 by 
\[
R(\zeta,C)=(C-\zeta)^{-1}.
\]
Clearly $R(\zeta,C)$ is well defined at any $\zeta$ not inside the spectrum of $C$. Moreover if $C$ depends $\mathbf{C}^k$ upon on a multivariate $\vec{F}$, using the derivative formula for matrix inversion, $R(\zeta,C)$ depends $\mathbf{C}^k$ upon $\vec{F}$ as well. The resolvent can be used in the Dunford-Taylor integral to give simple expression for functions of matrix $C$, see section I.5.6. For example in the computation of  ETKF,  we need to compute $(I+C)^{-1/2}$, which can be expressed as 
\[
(I+C)^{-1/2}=\frac{1}{2\pi i }\int_{\Gamma}(1+\zeta)^{-1/2}R(\zeta,C)d\zeta.
\]
Here $\Gamma$ is a closed loop in the positive half of the complex plain that encloses all eigenvalues of $C$. Then by the differentiability of $R(\zeta_,C)$,  it is very easy to see the following lemma
\begin{lem}
\label{lem:matrixhalf}
Let $C(\vec{F})$ be a symmetric semi positive definite matrix that depends $\mathbf{C}^k$ upon $\vec{F}\in \mathbb{R}^n$, then $(I+C(\vec{F}))^{-1/2}$ exists and depends $\mathbf{C}^k$ upon $\vec{F}$. 
\end{lem}

When both $C$ and its perturbation are  symmetric, the projection to the eigenspace of a particular eigenvalue $\lambda$, which is called the eigenprojection (Section I.5.3), can be defined through a contour integral of the resolvent using Cauchy's integral formula:
\[
P_\lambda=-\frac{1}{2\pi i}\int_{\Gamma_\lambda} R(\zeta,C)d\zeta,
\]
where $\Gamma_\lambda$ is a path in the complex half-plane $\{ z \in \mathbb{C} : {\rm{Re}}(z) > -1 \}$ that encloses $\lambda$ but not any other eigenvalues.  As a reminder, when the eigenvalue $\lambda$ splits into different branches after perturbation, the perturbed eigenprojection $P_\lambda (\vec{F})$ will be the sum of the eigenprojections of all branches, and therefore it is also known as the total projection of the $\lambda$-group (Section II.2.1). This is a slightly nasty case and we try to avoid it by picking a point where the spectrum is stable.  In particular, if $\lambda$ is a simple eigenvalue, $P_\lambda$ is the orthogonal projection to the eigenvector, and the splitting singularity does not exist for $\lambda$. Because $R(\zeta,C)$ is smooth or $\mathbf{C}^k$ w.r.t. $\vec{F}$ and the eigenvalues are continuous with respect to perturbations, the eigenprojection $P_\lambda$ is smooth or $\mathbf{C}^k$ w.r.t. $\vec{F}$. In particular, the directional derivative in a direction $f$  can be  computed as 
\begin{equation}
\label{eqn:P1}
\mathcal{D}_f P_\lambda= P_\lambda (\mathcal{D}_f C)[\sum_{\eta\neq \lambda}(\lambda-\eta)^{-1} P_\eta]
+ [\sum_{\eta\neq \lambda}(\lambda-\eta)^{-1} P_\eta] (\mathcal{D}_f C)P_\lambda.
\end{equation}
Here the summation of $\eta$ is over all spectrum of $C$ that is not $\lambda$. (See page 88 equation (2.14) because all the eigenvalues are semi simple when $C$ is symmetric, there $S$ is given by (5.28) and (5.32) on page 42). 

The perturbation to the eigenvectors is studied through transformation functions (Section II.4.2). Consider a perturbation $C(x)$ parametrized by a scalar $x$ such that $C=C(0)$, then the transformation matrix is defined through an ODE:
\begin{equation}
\label{eqn:ODE}
\dot{U}(x)=\sum_{\lambda}\dot{P}_\lambda(x) P_{\lambda}(x)U(x),\quad U(0)=I_d. 
\end{equation}
One can similarly define $U^{-1}(x)$ and consequently show that $U(x)P_\lambda(x)U^{-1}(x)=P_\lambda(0)$ for all eigenvalues $\lambda$ of $C$. Moreover, when $C$ and $M$ are symmetric, $U(x)$ is actually an orthogonal matrix (Section II.6.2), so if no eigenvalues splitting occur after the perturbation, then $U(x)$ actually provides a transformation  of the eigenvectors of $C$ to the ones of $C(x)$. Therefore, for symmetric perturbation $C(\vec{F})$ small enough, we can define a transformation matrix $U(\vec{F})$ as $U(1)$ with  $C_{\vec{F}}(x)=C+x(C(\vec{F})-C)$. If $C(\vec{F})$  depends $\mathbf{C}^k$ on $\vec{F}$, then because the coefficients of the ODE \eqref{eqn:ODE} depends $\mathbf{C}^k$ on $\vec{F}$, then so is the solution 
\[
U(\vec{F})=\exp\bigg(\int^1_0 \sum_{\lambda}\dot{P}_{\lambda,\vec{F}}(x) P_{\lambda,\vec{F}}(x)dx\bigg),\quad\text{with}\quad
P_{\lambda,\vec{F}}(x)=-\frac{1}{2\pi i}\int_{\Gamma_\lambda} R(\zeta,C_{\vec{F}}(x))d\zeta
\]
and $\dot{P}_{\lambda,\vec{F}}(x)$ given by \eqref{eqn:P1}. Notice that $C_{\vec{F}}(\epsilon x)=C_{\epsilon \vec{F}}(x)$, and $D_{\epsilon \vec{F}}C=\epsilon^{-1}D_{\vec{F}} C$ so $\dot{P}_{\lambda,\epsilon \vec{F}}=\epsilon^{-1}\dot{P}_{\lambda,\epsilon\vec{F}}$ using a change of variable formula we get
\begin{align*}
U(\epsilon\vec{F})
&=\exp\bigg(\int^1_0 \sum_{\lambda}\dot{P}_{\lambda,\epsilon\vec{F}}(x) P_{\lambda,\epsilon\vec{F}}(x)dx\bigg)\\
&=\exp\bigg(\int^\epsilon_0 \sum_{\lambda}\dot{P}_{\lambda,\vec{F}}(x) P_{\lambda,\vec{F}}(x)dx\bigg).
\end{align*}
So the directional derivative of $U(\vec{F})$ at $\vec{F}=0$ is
\begin{align}
\notag
\mathcal{D}_f U(\vec{F})\bigg|_{\vec{F}=0}&=\frac{d}{d\epsilon}\exp\bigg(\int^\epsilon_0 \sum_{\lambda}\dot{P}_{\lambda,f}(x) P_{\lambda,f}(x)dx\bigg)\bigg|_{\epsilon=0}=\sum_{\lambda}\dot{P}_{\lambda,f} P_{\lambda,f}\\
\label{eqn:perturbU}
&=\sum_{\lambda}\bigg[P_\lambda (\mathcal{D}_f C)[\sum_{\eta\neq \lambda}(\lambda-\eta)^{-1} P_\eta]
+ [\sum_{\eta\neq \lambda}(\lambda-\eta)^{-1} P_\eta] (\mathcal{D}_f C)P_\lambda\bigg]P_\lambda. 
\end{align}
Formula \eqref{eqn:perturbU} can also be presented in a more concrete matrix form through its image over eigenvectors. Suppose the eigenvalue decomposition of a symmetric matrix $C$ be given as $C=\Psi^T\Sigma\Psi$, we consider the perturbation of $\Psi(\vec{F})=\Psi U(\vec{F})^T$ near $\vec{F}^*$. Let $\psi_i$ be one of the left eigenvectors associated with the $i$-th eigenvalue $\lambda_l$, then $P_{\lambda_i}$ is give by matrix $\Psi_i^T\Psi_i$, where $\Psi_i$ is $\Psi$ with rows not associated with $\lambda_l$ removed.  Then using $\psi_i\Psi^T_j\Psi_j=1_{i=j}\psi_i$ we have
\begin{equation}
\label{eqn:perturbeigen}
\mathcal{D}_M\psi_i=
\psi_i (\mathcal{D}_f C)\sum_{j\neq i}(\lambda_i-\lambda_j)^{-1} \Psi^T_j\Psi_j.
\end{equation}
To conclude, we have shown that
\begin{lem}
\label{lem:eigentransform} 
Let $C(\vec{F})\in \mathbb{R}^{d\times d}$ be a symmetric semi positive definite matrix that depends $\mathbf{C}^1$ upon $\vec{F}\in \mathbb{R}^n$, then around any point $\vec{F}^*$, there is an orthogonal transformation map $U(\vec{F})$ that depends $\mathbf{C}^1$ upon $\vec{F}$ with directional derivative given by \eqref{eqn:perturbU} or matrix representation \eqref{eqn:perturbeigen}. It characterizes the perturbation of eigen-projection  by
\[
U(\vec{F})^TP_\lambda(\vec{F})U(\vec{F})=P_\lambda(\vec{F}^*).
\]
In particular, when $\lambda$ is a simple eigenvalue of $C(\vec{F}^*)$, $U(\vec{F})$ transforms the eigenvector of $C(\vec{F}^*)$ associated with $\lambda$ to an eigenvector of $C(\vec{F})$ associated with the perturbed $\lambda$.
\end{lem}
\section*{References}
\bibliographystyle{unsrt}
\bibliography{ref}
\end{document}